%% file: main.tex
\documentclass{sn-jnl}

\input{Definitions}

\usepackage{biblatex}
\begin{document}

\title{Geometric construction of Kashiwara crystals on multiparameter persistence}

\author[1]{Yasuaki Hiraoka}
        \affil[1]{Kyoto University Institute for Advanced Study, Kyoto University.  Inamori Research Institute for Science}
        \email{hiraoka.yasuaki.6z@kyoto-u.ac.jp}
\author*[2]{Kohei Yahiro}
        \affil[2]{Kyoto University Institute for Advanced Study, Kyoto University}
        \email{yahiro.kohei.5r@kyoto-u.ac.jp}

\abstract{We establish a geometric construction of Kashiwara crystals on the irreducible components of the varieties of multiparameter persistence modules. Our approach differs from the seminal work of Kashiwara and Saito, as well as subsequent related works, by emphasizing commutative relations rather than preprojective relations  for a given quiver. Furthermore, we provide explicit descriptions of the Kashiwara operators in the fundamental cases of 1- and 2-parameter persistence modules, offering concrete insights into the crystal structure in these settings.}

\keywords{Multiparameter persistence module, Kashiwara crystal, Irreducible component, Representation theory}

\pacs[MSC Classification]{16G20, 55N31}

\maketitle

\section{Introduction}

\subsection{Background of our study}

In recent years, persistent homology has grown rapidly and been adopted widely across multiple areas of mathematics \cite{cmvj,ams}. 
The concept of persistent homology was first introduced in \cite{elz}, where the authors investigated filtered simplicial complexes with one-parameter and developed several concepts to characterize topological persistence along filtration sequences.
Subsequently, in \cite{zc}, the algebraic structure now known as interval decompositions was clarified by reformulating persistent homology in terms of finitely generated modules over a principal ideal domain and examining their indecomposable decompositions. 
Then, persistence diagrams are defined as invariants given by the indecomposable decompositions. 
A connection with quiver representation theory was later established in \cite{dc}, which started the significant research activity using quiver representations to generalize persistence modules to those with multiple filtration parameters \cite{BL,eh,oudot}.

In these algebraic studies using quiver representations, one of the main goals is to find invariants which capture the features of given persistence modules. However, in multiparameter settings, most representation categories are classified as infinite representation type, 
making the indecomposable decomposition complicated in general. 
In response to this fundamental limitation, several research groups have proposed closely interrelated concepts such as interval approximations \cite{asashiba}, generalized persistence diagrams \cite{km}, and signed barcodes \cite{boo}, wherein an appropriate set of interval indecomposables can recover various invariants of the original persistence module.
While this approach avoids the difficulty of indecomposable decompositions, several critical issues remain unresolved, including the interpretation of negative multiplicities in approximations and the development of computationally efficient methods for determining these approximations.

Incidentally, throughout the historical development, geometric investigations of representation varieties defined on quivers have led to remarkable advances in quiver representation theory. Seminal contributions include Lusztig's introduction of canonical bases \cite{lusztig} and the geometric construction of Kashiwara crystal by Kashiwara-Saito \cite{KS}. Notably, in \cite{KS}, the Kashiwara crystal $B(\infty)$ of the negative half of quantum groups was constructed using irreducible components of the varieties of certain representations defined on a given quiver. 
This result demonstrates that irreducible decompositions possess rich geometric and combinatorial structures that enhance our understanding of quiver representation varieties. We also refer the paper \cite{GLS2,NT} for recent generalizations of this field. 

This observation motivates the present work. 
Instead of studying each persistence module, we study the set of all persistence modules from an algebro-geometric viewpoint. 
The representation variety of multiparameter persistence modules has a natural action of an algebraic group. The orbits under this action correspond bijectively to isomorphism classes of representations. 
In the one-parameter case, the orbits correspond bijectively to persistence diagrams and the number of orbits of this action is finite for each fixed dimension vector. 
In general multiparameter persistence, however, the number of orbits of this action may be infinite even for a fixed dimension vector if the base field is infinite, and the parametrization of orbits is complicated. 
Thus, in this paper we introduce a new geometric invariant of persistence modules that works in a multiparameter setting.
To this goal, inspired by \cite{KS}, we investigate irreducible decompositions of the representation varieties of multiparameter persistence modules.

\subsection{Our contribution}

We define $d$-parameter persistence modules (Definition \ref{defn:2D_pm}) to be representations of a bound quiver $\overrightarrow{G}_{m}$ (Definition \ref{defn:grid}), called the equioriented commutative $d$-grid of size $m=(m_1,\dots,m_d)$.
We demonstrate that these irreducible components form a Kashiwara crystal (Definition \ref{defn:crystal}) with the root system determined by $\overrightarrow{G}_{m}$. 
More precisely, let $I$ be the set of vertices in $\overrightarrow{G}_{m}$, $E_C(\nu)$ be the representation variety of $d$-parameter persistence modules with dimension vector $(\nu_i)_{i\in I}\in\N^I$, and  $B_C(\nu)$ be the set of all irreducible components of the variety $E_C(\nu)$.  Our main theorem is stated as follows. 

\begin{mainthm}[See Theorem \ref{thm:BC}, \ref{thm:ef_an} and \ref{thm:2x2ef} for precise statements]
    $B_C= \bigsqcup_\nu B_{C}(\nu)$ has a nontrivial structure of the Kashiwara crystal $(B_C, \weight, (\varepsilon_i)_{i \in I}, (\varphi_i)_{i \in I}, (\tilde e_i)_{i \in I}, (\tilde f_i)_{i \in I})$ with respect to the root system associated with the quiver $\overrightarrow{G}_{m}$.
\end{mainthm}

Here, the maps $(\weight, (\varepsilon_i)_{i \in I}, (\varphi_i)_{i \in I})$ and the Kashiwara operators $((\tilde e_i)_{i \in I}, (\tilde f_i)_{i \in I})$ are defined in Definition \ref{defn:epsilonphi} and \ref{defn:ef}.
Theorem \ref{thm:BC} proves that $(B_C, \weight, (\varepsilon_i)_{i \in I}, (\varphi_i)_{i \in I}, (\tilde e_i)_{i \in I}, (\tilde f_i)_{i \in I})$ actually becomes a crystal, whereas Theorem \ref{thm:ef_an} and \ref{thm:2x2ef} imply that this crystal is nontrivial. The proof of our Main Theorem fundamentally follows the original approach in \cite{KS}, but with an essential distinction in our setting. Specifically, although both Cartan matrices are identical, \cite{KS} examines the double quiver with preprojective relations, whereas we investigate the quiver itself with commutative relations. This difference 
causes non-equidimensional irreducible components of $E_C(\nu)$ (Remark \ref{ddv}). Furthermore, we cannot directly apply the original construction of the Kashiwara operators to our setting, as irreducible components exhibit different behaviors under pullback and pushforward operations via the variety of extensions.

In general, providing an explicit description of the crystal structure presents significant challenges. However, in the specific cases  $\overrightarrow{G}_{(n)}$ (corresponding to $A_n$ quiver with the equi-orientation) and $\overrightarrow{G}_{(2,2)}$, we have successfully derived  explicit descriptions of the Kashiwara operators (Theorem \ref{thm:ef_an} and Theorem \ref{thm:2x2ef}), as well as a complete parametrization of all irreducible components (Corollary \ref{cor:ex_irr}). These results provide concrete insights into the crystal structure within this particular setting.

As in the framework of Kashiwara and Saito \cite{KS}, there exists a $\ast$-crystal structure on $B_C$ (Theorem \ref{thm:BCstar}). In \cite{KS}, this $\ast$-crystal structure establishes the existence of the Kashiwara embedding and proves that their crystal is isomorphic to $B(\infty)$. In our setting, although we can also construct a $\ast$-crystal structure on $B_C$, it cannot be embedded into $B(\infty)$ (Corollary \ref{cor:BC2x2property}), which follows as a consequence of our explicit description of the crystal structure for $\overrightarrow{G}_{(2,2)}$.

\subsection{Future direction}

Firstly, we discuss a relation between our result and the study of representation theory of quantum groups.
Historically, in the original work of Kashiwara-Saito \cite{KS}, the crystal of the negative half of the quantum group associated with a quiver was geometrically constructed on $B(\infty)=\bigsqcup_\nu B(\infty,\nu)$ using the set $B(\infty,\nu)$ of irreducible components of the representation variety $\Lambda(\nu)$ ($\nu\in -{\rm Q}_+$). In this sense, quantum groups lie at the core of this problem. 
On the other hand, in this paper, we first established a geometric construction of the  crystal on the set $B_C=\bigsqcup_\nu B_C(\nu)$ of irreducible components of the representation variety $E_C(\nu)$ ($\nu\in -{\rm Q}_+$) for multiparameter persistence modules. 
Until now, we have not yet identified any relationships with quantum groups, except for Corollary \ref{cor:BC2x2property} (1), in our framework of multiparameter persistence modules. 
These relationships require thorough investigation to enhance our understanding of  algebraic structures underlying multiparameter persistence modules. 
As an initial direction of this subject, exploring the relationship 
between $\overrightarrow{G}_{(n)}$ and its preprojective version
or 
between $\overrightarrow{G}_{(2,2)}$ and the affine root system $\tilde A_3$ may yield valuable insights, given that the generalized Cartan matrices of them are identical. 

Secondly, we discuss relations between our work and studies on persistence modules themselves.  
The assignment to which irreducible component a multiparameter persistence module belongs gives an invariant. 
This is multi-valued since the irreducible components may intersect. 
It should be noted that it is finitely valued in each fixed dimension but coarser than orbit decompositions. 
However, this has a remarkable feature that its value is related to each other by Kashiwara operators. 
Kashiwara operators make a Kashiwara crystal into a crystal graph, which has a graph metric. We hope that this metric will be applicable to study the stability with respect to this invariant of multiparameter persistence modules. 

\subsection{Outline of this paper}

The structure of this paper is as follows. In Section \ref{sect:prel}, we review the definitions of Kashiwara crystals, quiver representations, multiparameter persistence modules, and the geometric construction of $B(\infty)$ developed in \cite{KS}. In Section \ref{sect:constBC}, we establish two crystal structures on the set $B_C$ of all irreducible components of the representation varieties of multiparameter persistence modules. In Section \ref{sect:An} and \ref{sect:2x2}, we examine the cases $\overrightarrow{G}_{(n)}$ and $\overrightarrow{G}_{(2,2)}$, respectively, in detail. In Example \ref{ex:1112}, we compute Kashiwara operators in a low dimension.
As a corollary of Theorem \ref{thm:2x2ef}, we demonstrate that, in the genuinely multiparameter persistence setting, the crystal $B_C$ cannot be embedded into the crystal $B(\infty)$, is not upper seminormal, and is connected (Corollary \ref{cor:BC2x2property}). 
In the appendix we give an explicit description of the general module of $\overrightarrow{G}_{2,2}$ in each irreducible component of $E_C(\nu)$.

\section{Preliminary}\label{sect:prel}

In this paper, we denote by $\Z=\{0,\pm 1, \pm 2,\dots\}$ and $\N=\{0,1,2,\dots\}$ the sets of integers and of natural numbers including 0, respectively, and fix a field $\Bbbk$.

\subsection{Crystals}\label{subsect:crystals}

Let us review the definition of Kac-Moody root systems following Kac \cite[\S 1.1]{K}. 
\begin{defn}\label{defn:cartan}\cite{K}
Let $I$ be a finite set. A \emph{generalized Cartan matrix} is a square matrix $A=(A_{ij})_{i,j\in I}$ with integral entries satisfying the following properties: 
    \begin{enumerate}
        \item $A_{ii} = 2$ for all $i$. 
        \item $A_{ij}$ is nonpositive if $i \neq j$.
        \item $A_{ij} \neq 0$ for $i \neq j$ implies $A_{ji} \neq 0$.
    \end{enumerate}
\end{defn}
For each generalized Cartan matrix $A$, we have a realization.  

\begin{defn}\label{defn:realization}\cite{K}
A \emph{realization} of an $n \times n$ generalized Cartan matrix $A$ is a triple $(\mathfrak h, \Pi, \check\Pi)$ consisting of a complex vector space $\mathfrak h$, $\Pi =\{ \alpha_1, \ldots, \alpha_n \} \subset \mathfrak h^{*}$, and 
$\check\Pi = \{ h_1, \dots, h_n  \}\subset \mathfrak h$ satisfying the following properties: 
\begin{enumerate}
    \item $\Pi$ and $\check\Pi$ are linearly independent.  
    \item $A_{ij} =\langle h_i, \alpha_j \rangle$. 
    \item $\dim_{\mathbb C} \mathfrak h = 2n - \rank A$.
\end{enumerate}
\end{defn}
It is known that we can construct a realization of $A$ such that there exists a free abelian group $P\subset \mathfrak h^*$ (weight lattice) with the following properties: 
\begin{enumerate}
    \item $\mathbb C \otimes_{\mathbb Z}P \cong \mathfrak h^*$.
    \item $\langle h_i, \lambda \rangle \in \Z$ for any $i \in I$ and $\lambda \in P$. 
    \item $\alpha_i \in P$ for any $i \in I$.
\end{enumerate}

Next we recall the definition of crystals and their morphisms due to Kashiwara \cite{Kashiwara}.
For the definition of quantum groups, we need to fix a weight lattice $P$.
\begin{defn}\label{defn:crystal}\cite{Kashiwara}
    A \emph{crystal} (or \emph{Kashiwara crystal}) associated with a generalized Cartan matrix $A$ and a weight lattice $P$ is a datum 
    $(B, \weight, (\varepsilon_i)_{i \in I}, (\varphi_i)_{i \in I}, (\tilde e_i)_{i \in I}, (\tilde f_i)_{i \in I})$ 
    consisting of a set $B$ and maps $\weight: B\rightarrow P$, 
    $\varepsilon_i, \varphi_i: B\rightarrow \Z\sqcup\{-\infty\}$, and 
    $\tilde e_i, \tilde f_i: B\rightarrow B \sqcup \{ 0 \}$ satisfying the following conditions: 
    \begin{enumerate}
        \item $ \varphi_i(b) = \varepsilon_i(b) + \langle h_i , \weight(b) \rangle$. 
        \item If $b \in B$ and $\tilde e_i b \in B$, then we have 
        $\weight(\tilde e_i b) = \weight(b) + \alpha_i$, 
        $\varepsilon_i(\tilde e_i b) = \varepsilon_i(b) - 1 $ 
        and 
        $\varphi_i(\tilde e_i b) = \varphi_i(b) + 1 $. 
        \item If $b \in B$ and $\tilde f_i b \in B$, then we have 
        $\weight(\tilde f_i b) = \weight(b) - \alpha_i$, 
        $\varepsilon_i(\tilde f_i b) = \varepsilon_i(b) + 1 $ 
        and 
        $\varphi_i(\tilde f_i b) = \varphi_i(b) - 1 $. 
        \item For $b, b' \in B$ and $i \in I$, 
        $b = \tilde e_i b'$ if and only if 
        $\tilde f_i b = b'$.
        \item For $b \in B$, if $\varphi_i(b)=-\infty$ then $\tilde e_i b = \tilde f_i b = 0$. 
    \end{enumerate}
    The maps $\tilde e_i$ and $ \tilde f_i$ are called the \emph{Kashiwara operators}.
\end{defn}
The notion of crystals was originally introduced to describe good bases of representations of quantum groups (see \cite{Kashiwara}). 

\begin{defn}\label{defn:crystalmor}\cite{Kashiwara}
    Let $B_1, B_2$ be crystals. 
    A \emph{morphism} of crystals is a map $\rho: B_1 \to B_2 \sqcup \{ 0 \}$ satisfying the following properties: 
    \begin{enumerate}
        \item If $b \in B_1$ satisfies $\rho(b) \neq 0$, 
        we have $\weight(b) = \weight(\rho(b)) $, $\varepsilon_i(b) = \varepsilon_i (\rho(b)) $,  $\varphi_i(b) = \varphi_i (\rho(b))$.
        \item If $b \in B_1$ satisfies $\rho(\tilde e_i(b)) \neq 0$ and $\rho(b) \neq 0$, 
        we have $\tilde e_i \rho(b) = \rho(\tilde e_i b) $.
        \item If $b \in B_1$ satisfies $\rho(\tilde f_i(b)) \neq 0$ and $\rho(b) \neq 0$, 
        we have $\tilde f_i \rho(b) = \rho(\tilde f_i b) $.
    \end{enumerate}
\end{defn}

\begin{defn}\label{defn:crystal_graph}
The \emph{crystal graph} associated with a crystal $B$ is a colored oriented graph whose vertex set is given by  $B$ and an $i$-colored oriented edge from $b$ to $b'$ for $b,b'\in B$ and $i\in I$ is assigned if $b'=\tilde f_i b$. A crystal $B$ is called \emph{connected} if the underlying graph of corresponding crystal graph is connected.
\end{defn}

\subsection{Quivers and their representations}

In this subsection, we review the definition of bound quivers and its representations. We also introduce the bound quiver $\overrightarrow{G}_{m}$ on which multiparameter persistence modules are defined in this paper.

\begin{defn}\label{defn:quiver}
    A \emph{quiver} is a quadruple $(I, E, s, t)$ consisting of a set $I$ of vertices, a set $E$ of arrows, and two maps $s$, $t: E\rightarrow I$ that assign the source and target of arrows, respectively. 
    A \emph{path} $\mu$ of length $n\in\N$ is defined by a sequence $\tau_n\cdots\tau_1$
    of arrows satisfying $t(\tau_i)=s(\tau_{i+1})$ for $i=1,\dots,n-1$. 
    Here, $s(\tau_1)$ and $t(\tau_n)$
    are called the source and target vertices of the path $\mu$, respectively. Multiple paths $\mu_1,\dots,\mu_m$ are called \emph{parallel} if their source and target vertices coincide. 
    A \emph{relation} $\rho$ in $Q$ is a formal sum $\rho=\sum_{i=1}^m c_i \mu_i$ of parallel paths $\mu_i$, where each path $\mu_i$ is of length at least 2 and $c_i\in\Bbbk$. A pair $(Q, R)$ of a quiver $Q$ and a set $R$ of relations is called a \emph{bound quiver}. 
\end{defn}

In this paper, we always assume that the quiver $Q$ has no edge loops, i.e., $\{\tau\in E~|~s(\tau)=t(\tau)\}=\emptyset$. 

\begin{defn} 
A \emph{representation} $M=(M_i,f_\tau)_{i\in I, \tau\in E}$ of a quiver $Q$ is defined by the following data:
\begin{enumerate}
\item A (finite dimensional) $\Bbbk$-vector space $M_i$ for each vertex $i\in I$.
\item A $\Bbbk$-linear map $f_\tau: M_i\rightarrow M_j$ for each arrow $\tau:i\to j$.
\end{enumerate}
The evaluation of $M$ on the path $\mu=\tau_n\cdots \tau_1$ is defined by the composition $f_\mu=f_{\tau_n}\circ \dots\circ f_{\tau_1}$. 
This definition is extended to a relation $\rho =\sum_{i=1}^m c_i\mu_i$ by 
$f_\rho=\sum_{i=1}^m c_i f_{\mu_i}$.
A \emph{representation of a bound quiver} $(Q,R)$ is a representation of $Q$ satisfying $f_\rho=0$ for all $\rho\in R$.
For a given representation $M$, the \emph{dimension vector} is defined by $(\dim M_i)_{i \in I} \in \mathbb{N}^{I}$.
For a vertex $i \in I$, let $S_i$ be the simple representation of $(Q,R)$ with dimension vector $(\delta_{ij})_{j \in I}$ and with all $\Bbbk$ linear maps associated with arrows being zero. 

\end{defn}

\begin{defn}\label{defn:quiverCartan}
    A symmetric generalized Cartan matrix $A=A(Q)$
    associated with a quiver $Q=(I,E,s,t)$  is defined  by 
    \begin{align}
        &A_{ii} =2,  \\
        &A_{ij} = -\#\{\tau\in E~|~s(\tau)=i, t(\tau)=j\} - 
        \#\{\tau\in E~|~s(\tau)=j, t(\tau)=i\}
        \text{ if  $i\neq j$}.
    \end{align}
\end{defn}

For $\ell \in \mathbb N$, we write $[1, \ell] = \{ i \in \mathbb N ~\vert~ 1 \leq i \leq \ell \}$.

\begin{defn}\label{defn:grid}(\cite{asashiba})
    For $m=(m_1,\dots,m_d)\in \N^d$, the \emph{equioriented commutative $d$-grid} $\overrightarrow{G}_{m}$ is defined as a bound quiver $(Q,C)$, where $Q=(I,E,s,t)$ is given by $I= [1, m_1] \times \dots \times [1, m_d]$, 
    \begin{align*}
    E = \{ (i, j) \in I \times I ~\vert&~i=(i_1,\dots,i_d), j=(j_1,\dots,j_d),\\
     &~\text{and $i_{\ell}+1=j_\ell$, $i_k=j_k$  ($k\neq\ell$) for $\ell=1,\dots,d$
     }
     \}, 
    \end{align*}
    and
    $s$ and $t$ are the first and the second projections from $E \subset I \times I$ to $I$, endowed  with 
    the commutative relations $C$ that any two parallel paths are identical.
\end{defn}

For example, $\overrightarrow{G}_{( m_1,m_2)}$ is illustrated as
\begin{center}
\begin{tikzcd}
    (1,m_2) \arrow[r, ]& (2,m_2) \arrow[r, ] & \cdots \arrow[r, ]& (m_1,m_2)\\
    \vdots \arrow[u, ]\arrow[ur, phantom, ""] & \vdots \arrow[u, ] &  \arrow[ur, phantom, "\rotatebox{135}{\vdots}"]& \vdots \arrow[u, ]\\
    (1,2) \arrow[u, ] \arrow[r, ]\arrow[ur, phantom, ""]  & (2,2) \arrow[u, ] \arrow[ur, phantom, "\rotatebox{135}{\vdots}"]  \arrow[r, ]& \cdots \arrow[r, ]& (m_1,2) \arrow[u, ]\\
    (1,1) \arrow[r, ] \arrow[u, ] \arrow[ur, phantom, "\circlearrowleft"] & (2,1) \arrow[u, ]\arrow[r, ] \arrow[ur, phantom, ""] & \cdots \arrow[r, ] \arrow[ur, phantom, ""] & (m_1,1) \arrow[u, ]
\end{tikzcd},
\end{center}
where the commutative relations $C$ are generated by
    \[
    ((i_1,i_2+1),(i_1+1, i_2+1))((i_1,i_2),(i_1, i_2+1)) - ((i_1+1,i_2),(i_1+1, i_2+1))((i_1,i_2),(i_1+1, i_2))
    \]
    for all $1 \leq i_1 \leq m_1-1$, $1 \leq i_2 \leq m_2-1$.

\begin{defn}\label{defn:2D_pm}
For $m=(m_1,\dots,m_d)\in \N^d$, a representation of $\overrightarrow{G}_{m}$ is called a \emph{$d$-parameter persistence module}. 
\end{defn}

\subsection{Geometric construction of $B(\infty)$}

Let $Q=(I,E,s,t)$ be a quiver. 
We associate  to $Q$ the symmetric generalized Cartan matrix $A(Q)$, the quantum group $U_{q}=U_{q}(\mathfrak{g}(A(Q)))$, and the crystal basis  $B(\infty)$ of the negative half $U_{q}^{-}$ of $U_{q}$ \cite{Kashiwara}. 
One of the main results of the paper by Kashiwara and Saito \cite{KS} establishes a geometric construction of the crystal $B(\infty)$ by irreducible components of certain representation varieties. 
In this subsection, we briefly recall its construction, following sections 3, 4, and 5 in \cite{KS}. Here, we basically use the original notations.


Let $\bar Q = (I, E \sqcup \bar E, s, t)$ be the double of the quiver $Q$ defined by adding opposite arrows $\bar E = \left\{ \bar \tau ~\vert~ \tau \in E \right\}$ with $s(\bar \tau) =t(\tau) $ and $t(\bar \tau) = s(\tau)$. 
To each dimension vector $(\nu_i)_{i \in I}$, let us assign a weight $\nu=-\sum_{i \in I} \nu_i \alpha_i \in - \sum_{i \in I} \mathbb{N}\alpha_i =: -\textrm{Q}_+ \subset P$.
We also define the preprojective relation 
\begin{equation}\label{eq:preprojective}
\mu_i=\sum_{\tau \in E, s(\tau)= i} \bar \tau \tau - \sum_{\tau \in E, t(\tau)= i} \tau \bar \tau
\end{equation}
on $\bar{Q}$ for each $i\in I$. 
Then, for each $\nu \in -{\rm Q}_+$, the representation variety of the bound quiver $(\bar Q, \left\{ \mu_i ~|~ i \in I \right\})$ is defined as follows 
\[
\Lambda(\nu)= \left\{ 
\left. B \in \prod_{\tau \in E\sqcup\bar E} \Hom(\Bbbk^{\nu_{s(\tau)}},\Bbbk^{\nu_{t(\tau)}}) ~\right\vert~
B_{\mu_i} = 0 \text{ for all } i \in I \text{ and } B \text{ is nilpotent} 
\right\}.
\]
Here, a representation $B$ is called nilpotent if there exists $n\in \N$ such that we have $B_\mu = 0$ for all paths $\mu$ of length $n$.
Let
$B(\infty, \nu) = \irr \Lambda(\nu)$ be the set of all irreducible components in the representation variety $\Lambda(\nu)$.

\begin{thm}\label{thm:KS}
\begin{enumerate}
    \item \cite[Theorem 5.2.6.]{KS} There is a crystal structure on $\bigsqcup_{\nu \in -{\rm Q}_+} B(\infty, \nu)$.
    \item \cite[Theorem 5.3.2.]{KS} The crystal $\bigsqcup_{\nu \in -{\rm Q}_+} B(\infty, \nu)$ is isomorphic to $B(\infty)$. 
\end{enumerate}
\end{thm}

We now outline the proof of this theorem. 
Recall that a crystal structure on $\bigsqcup_{\nu \in -\textrm{Q}_+} B(\infty, \nu)$ is given by maps $\weight, (\varepsilon_i)_{i \in I}, (\varphi_i)_{i \in I}, (\tilde e_i)_{i \in I}$ and $ (\tilde f_i)_{i \in I}$. 
For $\Lambda \in B(\infty, \nu)$, the map $\weight: \bigsqcup_{\nu \in -\textrm{Q}_+} B(\infty, \nu) \to P$ is defined by $\weight(\Lambda) = \nu$. 
For $B \in \Lambda$, we define the maps
$\varepsilon_i(B) = 
\dim \coker \left((B_\tau)_{\tau\in E\sqcup \bar E, t(\tau)=i}: 
\bigoplus_{\tau \in E \sqcup \bar E, t(\tau)= i} \Bbbk^{\nu_{s(\tau)}} 
\to \Bbbk^{\nu_i} \right)$,  
$\varepsilon_i(\Lambda) = \min_{B \in \Lambda} \varepsilon_i(B)$, and $\varphi_i(\Lambda) = \varepsilon_i(\Lambda) + \langle h_i , \weight(\Lambda) \rangle$. 
Here, we remark that taking the minimum in $\varepsilon_i(\Lambda)$ corresponds to evaluating $\varepsilon_i(B)$ at a generic point $B\in \Lambda$.

For the definition of $\tilde e_i$ and $\tilde f_i$, we need a variety parameterizing extensions of representations. 
For $\bar \nu, \nu' \in -\textrm{Q}_+$, let us define $\Lambda'(\bar \nu, \nu')$ to be the variety consisting of 
triples $B \in \Lambda(\bar \nu + \nu')$, $\bar \phi=(\bar \phi_i)_{i \in I}$ and $\phi'=(\phi'_i)_{i \in I}$ which give an exact sequence
\[
0 \to 
\Bbbk^{\bar \nu_{i}}
\xrightarrow{\bar \phi_i}
\Bbbk^{\bar \nu_i+\nu'_{i}}
\xrightarrow{\phi'_i}
\Bbbk^{\nu'_{i}} 
\to 0 
\]
for each $i$ satisfying $\image \bar\phi_{t(\tau)} \supset \image B_{\tau} \circ \bar\phi_{s(\tau)}$ for all $\tau \in E\sqcup \bar E$. 
Then $(B, \bar \phi, \phi') \in \Lambda'(\bar \nu, \nu')$ gives a subrepresentation $\bar B \in \Lambda(\bar \nu)$ induced by $\bar \phi$ and a quotient representation $B' \in \Lambda(\nu')$ induced by $\phi'$, and hence this gives a morphism 
$q_1:\Lambda'(\bar \nu, \nu') \to \Lambda(\bar \nu) \times \Lambda(\nu'), (B, \bar \phi, \phi') \mapsto (\bar B, B')$. 
Let $q_2$ be the morphism 
$q_2:\Lambda'(\bar \nu, \nu') \to \Lambda(\bar \nu + \nu'), (B, \bar \phi, \phi') \mapsto B$. 

For $i \in I$ and $c \in \mathbb{N}$, let $B(\infty, \nu)_{i,c}$ be the subset of $B(\infty, \nu)$ consisting of $\Lambda$ satisfying $\varepsilon_i(\Lambda)=c$. 
We have a decomposition $B(\infty, \nu) = \sqcup_{c \in \mathbb{N}} B(\infty, \nu)_{i,c} $. 
The Kashiwara operators will be defined on each $B(\infty, \nu)_{i,c}$. 
Let $\Lambda(\nu)_{i,c}$ be the subset of $\Lambda(\nu)$ consisting of $B$ satisfying  $\varepsilon_i(B)=c$. 
Then, we have a correspondence
\[
\Lambda(\nu )_{i,0} \xleftarrow{q_1} \Lambda'(\nu, - c \alpha_i)_0  \xrightarrow{q_2} \Lambda(\nu-c\alpha_i)_{i,c},
\]
where $\Lambda'(\nu, - c \alpha_i)_0$ is defined by $\Lambda'(\nu, - c \alpha_i)_0:=q_1^{-1}(\Lambda(\nu )_{i,0})=q_2^{-1}(\Lambda(\nu-c\alpha_i)_{i,c})$.
Lemma 5.2.3 of \cite{KS} shows that $q_1$ and $q_2$ are smooth morphisms with connected fibers, and as a consequence, this correspondence gives a bijection $\tilde f^{c}_{i}: B(\infty, \nu)_{i, 0} \to B(\infty, \nu -c \alpha_i)_{i, c}$ (\cite[Proposition 5.2.4]{KS}). 

Finally, we define Kashiwara operators as follows. 
The map $\tilde e_i : B(\infty, \nu)_{i,c} \to B(\infty, \nu + \alpha_i)_{i,c-1} \sqcup \{ 0 \}$ is defined to be $\tilde e_i \Lambda = 0$ if $c=0$ and $\tilde e_i \Lambda = \tilde f^{c-1}_{i} \circ (\tilde f^{c}_{i})^{-1} \Lambda$ if $c > 0$. 
The map $\tilde f_i : B(\infty, \nu)_{i,c} \to B(\infty, \nu - \alpha_i)_{i,c+1} $ is defined by $\tilde f_i \Lambda = \tilde f^{c+1}_{i} \circ (\tilde f^{c}_{i})^{-1}\Lambda$. We can check that this construction satisfies Definition \ref{defn:crystal}, proving Theorem \ref{thm:KS} (1). To prove Theorem \ref{thm:KS} (2), Kashiwara and Saito used a geometric construction of the Kashiwara operators $\tilde e^*_i$ and $\tilde f^*_i$,  which are defined using an automorphism $*$ of $U_q$ and a characterization of the crystal $B(\infty)$ in terms of $\tilde e^*_i$ and $\tilde f^*_i$. 
In this process, $\Lambda'(- c \alpha_i, \nu)$ is used instead of $\Lambda'(\nu, - c \alpha_i)$.

\section{Crystal structures on multiparameter persistence modules}\label{sect:constBC}

In this section, we study representations on the equioriented commutative $d$-grid $\overrightarrow{G}_{m}=:((I,E,s,t), C)$ (Definition \ref{defn:grid}) for $m=(m_1,\dots,m_d)$. 
Associated with $(I, E)$, we have a generalized Cartan matrix, its realization and weight lattice $P$ as in \S \ref{subsect:crystals}.  

Unlike the approach of studying representations on the double of the quiver in \cite{KS}, we investigate the following.

\begin{defn}\label{defn:modulidef}
For $\nu = -\sum_{i \in I} \nu_i \alpha_i \in -{\rm Q}_+$, a representation variety $E_C(\nu)$ of $\overrightarrow{G}_{m}$ is defined by
\[E_C(\nu):= \left\{\left. f=(f_{\tau}) \in \prod_{\tau \in E} \Hom(\Bbbk^{\nu_{s(\tau)}},\Bbbk^{\nu_{t(\tau)}}) ~ \right| ~ f_\rho = 0 ~{\rm for} ~\forall\rho\in C \right\}. \]
The set of all irreducible components of the algebraic variety $E_C(\nu)$ and the disjoint union of them with respect to $\nu$ are denoted by 
\begin{align*}
B_{C}(\nu) &= \irr E_C(\nu), \\
B_C & = \bigsqcup_\nu B_C(\nu). 
\end{align*}
\end{defn}

The variety $E_C(\nu)$ is the moduli variety of all $d$-parameter persistence modules of the dimension vector $(\nu_i)_{i\in I}$. 
Since the condition that $(f_{\tau})_{\tau\in E}$ satisfies $f_\rho = 0$ for all $\rho\in C$ is described by polynomial equations, $E_C(\nu)$ is an affine algebraic variety. 
We remark that $E_C(\nu)$ reduces to merely a vector space in the special case of one-parameter persistence modules (corresponding to $d=1$). 

The purpose of this paper is to construct a crystal structure $(B_C, \weight, (\varepsilon_i)_{i \in I}, (\varphi_i)_{i \in I}, (\tilde e_i)_{i \in I}, (\tilde f_i)_{i \in I})$ on $B_C$ in a similar way to that by Kashiwara and Saito \cite{KS}. 
We first define the maps $\weight$, $(\varepsilon_i)_{i \in I}$ and $ (\varphi_i)_{i \in I}$ as follows. The Kashiwara operators $(\tilde e_i)_{i \in I}$ and $(\tilde f_i)_{i \in I}$ are defined in Definition \ref{defn:ef}. 

\begin{defn}\label{defn:epsilonphi}\hfill
    \begin{enumerate} 
    \item
    A function $\weight: B_C \to P$ is defined as $\weight(\Lambda) := \nu$ for $\Lambda \in B_C(\nu)$. 
    \item For $f \in E_C(\nu)$, $\varepsilon_i(f) := 
    \dim \coker ( (f_{\tau})_{\tau \in E, t(\tau) =i}:
    \bigoplus_{\tau \in E, t(\tau) =i} \Bbbk^{\nu_{s(\tau)}} \to \Bbbk^{\nu_i})$. 
    \item 
    For each $i \in I$, a function $\varepsilon_i: B_C \to \Z\sqcup\{-\infty\}$ is defined as $\varepsilon_i(\Lambda) := \min_{f \in \Lambda} \{ \varepsilon_i(f) \}$.
    \item 
    For each $i \in I$, a function $\varphi_i: B_C \to \Z\sqcup\{-\infty\}$ is defined as $\varphi_i(\Lambda) := \varepsilon_i(\Lambda) + \langle h_i, \weight(\Lambda) \rangle$. 
    \end{enumerate}
\end{defn}

In order to define the Kashiwara operators, we introduce the following variety.

\begin{defn}\label{defEprime}
    For $\bar\nu,\nu'\in -{\rm Q}_+$, let $E'_C(\bar\nu,\nu')$ be the subvariety of 
    \[
        \prod_{i \in I} \Hom(\Bbbk^{\bar\nu_i}, \Bbbk^{\bar\nu_i+\nu'_{i}}) \times 
        \prod_{i \in I} \Hom(\Bbbk^{\bar\nu_i+\nu'_{i}}, \Bbbk^{\nu'_{i}}) \times 
        \prod_{\tau \in E} \Hom(\Bbbk^{\bar\nu_{s(\tau)}+\nu'_{s(\tau)}}, \Bbbk^{\bar\nu_{t(\tau)}+\nu'_{t(\tau)}}) 
    \]
    consisting of $((\bar\phi_i)_{i \in I}, (\phi'_i)_{i \in I}, (f_{\tau})_{\tau\in E} ) $
    satisfying the following properties.
    \begin{enumerate} 
        \item 
        For each $i \in I$, 
        \[ 
        0 \to \Bbbk^{\bar\nu_{i}}
        \xrightarrow{\bar\phi_i}  
        \Bbbk^{\bar\nu_i+\nu'_{i}}
        \xrightarrow{\phi'_i} 
        \Bbbk^{\nu'_{i}} \to 0 
        \]
        is a short exact sequence. 
        \item $(f_{\tau})_{\tau\in E}$ satisfies $C$. 
        \item For each $\tau \in E$, we have $\image \bar\phi_{t(\tau)} \supset \image f_{\tau} \circ \bar\phi_{s(\tau)} $.
    \end{enumerate}  
\end{defn}

We may think of $E'_C(\bar\nu,\nu')$ as the moduli variety consisting of all extensions of representations of $\overrightarrow{G}_{m}$ with fixed dimension vectors. It should  be noted that 
$(f_\tau)_{\tau\in E}$ induces a subrepresentation $(\bar f_{\tau})_{\tau \in E} \in E_C(\bar\nu)$ and a quotient representation 
$(f'_{\tau})_{\tau \in E} \in E_C(\nu')$ 
due to (3) in Definition \ref{defEprime}. From this assignment, we define 
morphisms $\varpi_1 : E'_C(\bar\nu,\nu') \to E_C(\bar\nu) $ by 
$((\bar\phi_i)_{i \in I}, (\phi'_i)_{i \in I}, (f_{\tau})_{\tau \in E} )\mapsto (\bar f_{\tau})_{\tau \in E}$, 
and 
$\varpi_2 : E'_C(\bar\nu,\nu') \to E_C(\bar\nu+\nu') $ by 
$((\bar\phi_i)_{i \in I}, (\phi'_i)_{i \in I}, (f_{\tau})_{\tau \in E} ) \mapsto (f_{\tau})_{\tau \in E}$. 

From here to Lemma \ref{prop:ef}, we fix $i \in I$ and let $\nu +\alpha_i \in -{\rm Q}_+$.
The correspondence to construct the Kashiwara operators $\tilde e_i, \tilde f_i$ is given by 
\[
E_C(\nu+\alpha_i) \xleftarrow{\varpi_1} E'_C(\nu+\alpha_i, -\alpha_i) \xrightarrow{\varpi_2} E_C(\nu).
\]
In this setting, for $((\bar\phi_j)_{j \in I}, (\phi'_j)_{j \in I}, (f_{\tau})_{\tau \in E}) \in E'_C(\nu+\alpha_i, -\alpha_i)$ 
and for each $j \in I \setminus \{ i \}$, 
$\bar\phi_j$ is an isomorphism and $\phi'_j$ is zero. 

The basic idea to construct the Kashiwara operators is to define 
$\tilde f_i \bar \Lambda := \overline{\varpi_2 \varpi^{-1}_1(\bar \Lambda)}$ 
and 
$\tilde e_i \Lambda :=\overline{\varpi_1 \varpi^{-1}_2(\Lambda)}$
for irreducible components $\bar \Lambda \in B_C(\nu+\alpha_i)$ and $\Lambda \in B_C(\nu)$, where $\overline{\varpi_2 \varpi^{-1}_1(\bar \Lambda)}$ and $\overline{\varpi_1 \varpi^{-1}_2(\Lambda)}$ mean the closures of $\varpi_2 \varpi^{-1}_1(\bar \Lambda)$ and $\varpi_1 \varpi^{-1}_2(\Lambda)$, respectively.  
However, these are not irreducible in general. 
In order to obtain irreducible algebraic sets, we need to replace the  irreducible components $\bar\Lambda$
and $\Lambda$ by suitable open subsets. 
This is the content of Proposition \ref{prop:varpi} below. To state and prove Proposition \ref{prop:varpi}, we introduce the following preparatory definition and lemmas. Here, we assign a lexicographical order on $I$ and denote it by $i<_{\text{lex}}j$ for $i,j\in I$.

\begin{defn}\label{defFib}
\hfill
\begin{enumerate}
\item $\out_1(i):=\{j\in I~|~(i,j)\in E\}$.
\item $\out_2(i):=\{k\in I~|~(j_1,k),(j_2,k)\in E~\text{for $j_1,j_2\in \out_1(i)$ and $j_1\neq j_2$}\}$.
\item For $j_1,j_2\in \out_1(i)$ with $j_1<_\text{lex}j_2$, a vertex $k(j_1,j_2)\in\out_2(i)$ is defined by the unique vertex satisfying $(j_1,k(j_1,j_2)), (j_2,k(j_1,j_2))\in E$. Note that any $k\in \out_2(i)$ is uniquely expressed as $k(j_1,j_2)$ by $j_1,j_2\in\out_1(i)$.

\item A linear map $F_i(f):\oplus_{j\in\out_1(i)}\Bbbk^{\nu_j}\rightarrow\oplus_{k\in \out_2(i)}\Bbbk^{\nu_k}$ is defined for each direct summand 
$F_i(f)_{jk}: \Bbbk^{\nu_j} \to \Bbbk^{\nu_k} $          
by 
\[F_i(f)_{jk}
=
\begin{cases}
f_{(j,k(j,j'))} & \text{if there is $j' \in \out_1(i)$ with $j<_{\text{lex}}j'$ and $k(j,j') = k$},\\
-f_{(j,k(j',j))} & \text{if there is $j' \in \out_1(i)$ with $j'<_{\text{lex}}j$ and $k(j',j) = k$},\\
0 & \text{otherwise}.
\end{cases}
\]
Here, in the case of $|\out_1(i)|=1$, the linear map is understood as $F_i(f)=0$.

\item $\Exttilde_i(f):=\kernel F_i(f)$.
\end{enumerate}
\end{defn}

\begin{rem}
    Let $\cdots \to Q_2 \to Q_1 \to Q_0 \to S_i \to 0$ be a minimal projective resolution of the simple module $S_i$. 
    Then we have $Q_0 \cong P_i$, $Q_1 \cong \sum_{j \in \out_{1}(i)} P_j$ and $Q_2 \cong \sum_{j \in \out_{2}(i)} P_j$, where $P_j$ is a projective cover of $S_j$ for $j \in I$. 
    This leads to an isomorphism $\Exttilde_i(f) \cong \ker(\Hom_{\overrightarrow{G}_{m}}(Q_{1}, f) \to \Hom_{\overrightarrow{G}_{m}}(Q_{2}, f))$ and hence a surjection $\Exttilde_i(f) \to {\rm Ext}^1_{\overrightarrow{G}_{m}}(S_i, f)$.
\end{rem}

\begin{lem}\label{SES}
    For $m,n \in \mathbb{N}$, let $S(m,n)$ be a subvariety of 
    $\Hom (\Bbbk^{m}, \Bbbk^{m+n}) \times \Hom (\Bbbk^{m+n}, \Bbbk^{n})$
    consisting of pairs $(\alpha, \beta)$ of maps such that 
    \[
    0 \to \Bbbk^m \xrightarrow{\alpha} \Bbbk^{m+n} \xrightarrow{\beta} \Bbbk^{n} \to 0
    \]
    is a short exact sequence. 
    Then $S(m,n)$ is a homogeneous variety under the action of $\GL_{m+n}$ given by $g\cdot (\alpha, \beta) =(g \circ \alpha , \beta \circ g^{-1})$ for $g\in\GL_{m+n}$. In particular $S(m,n)$ is smooth and connected. 
\end{lem}
\begin{proof}
This is proved by the Gaussian elimination of matrices. 
\end{proof}

\begin{lem}\label{lem:smooth}
Let $\varpi: X\rightarrow Y$ be a morphism on algebraic varieties. Suppose that, for any point $y\in Y$, there exists an open neighborhood $V$ of $y$ and a smooth variety $Z$ such that the diagram
\[
\begin{tikzcd}
        \varpi^{-1}(V) \arrow[rightarrow,"\varpi|_{\varpi^{-1}(V)}"]{rr} \arrow[leftarrow,"q","\cong"']{d}& &  V\\
        Z\times V \arrow[rightarrow,"p"']{urr} & &
\end{tikzcd}
\]
commutes, where $p$ is the projection and $q$ is an isomorphism. Then, $\varpi$ is smooth. 
\end{lem}
\begin{proof}
The statement easily follows from Proposition 10.1 (b) in \cite{Hartshorne}.
\end{proof}

\begin{prop}\label{prop:varpi} \hfill
    \begin{enumerate}
    \item The morphism $\varpi_1$ is surjective and $\varpi_1^{-1}(\bar f)$ is a vector bundle over $ S(\nu_i-1, 1) \times \prod_{j \neq i} \GL_{\nu_j}$ whose fibers are isomorphic to $\Exttilde_i(\bar f)$.
    
    \item Let $\bar \Lambda \in B_C(\nu+\alpha_i)$. Define $\bar \Lambda^{1}$ to be the open subset of $\bar \Lambda$ consisting of points $\bar f$ at which the function $\bar f \mapsto \dim \Exttilde_i(\bar f)$ takes the minimum value in $\bar\Lambda$. 
    Then the restriction of $\varpi_1$ to $\varpi^{-1}_1(\bar \Lambda^{1})$ is a smooth surjective morphism to $\bar \Lambda^{ 1}$ with connected fibers. 
    
    \item The image of the morphism $\varpi_2$ is the closed subset of $E_C(\nu)$ consisting of points $f$ with $\varepsilon_i(f)>0$. 
    For $f \in E_C(\nu)$ with $\varepsilon_i(f)>0$, $\varpi^{-1}_2(f)$ is a principal $\GL_{\nu_i-1}$ bundle over $\prod_{j \neq i} \GL_{\nu_j}\times \left((\Bbbk^{\varepsilon_i(f)})^* \setminus \{ 0 \} \right)$. 
    
    \item Assume that $\Lambda \in B_C(\nu)$ satisfies $\varepsilon_i(\Lambda)>0$. 
    Define $\Lambda^{ 2}$ to be the open subset of $\Lambda$ consisting of points at which the function $\varepsilon_i$ takes the minimum value in $\Lambda$.
    Then the restriction of $\varpi_2$ to $\varpi^{-1}_2(\Lambda^{2})$ is a smooth surjective morphism to $\Lambda^{ 2}$ with connected fibers. 
    \end{enumerate}
\end{prop}

\begin{proof}
    (1) Let $\bar f = (\bar f_{\tau})_{\tau \in E} \in E_C(\nu+\alpha_i)$. The fiber $\varpi^{-1}_1(\bar f) $  consists of $( (\bar\phi_j)_{j \in I}, (\phi'_j)_{j \in I}, (f_{\tau})_{\tau \in E})$ 
    such that the restriction of $(f_{\tau})_{\tau \in E}$ via $(\bar\phi_j)_{j \in I}$ equals $(\bar f_{\tau})_{\tau \in E}$. 
    It is obvious to have $(\bar\phi_i, \phi'_i)\in S(\nu_i-1,1)$.   
    For $j \in I$ with $j \neq i$, $\bar\phi_j$ is any automorphism of $\Bbbk^{\nu_j}$, hence we have $\prod_{j \neq i} \GL_{\nu_j}$ in the fiber. 
    Thus, we obtain a morphism 
    $p: \varpi^{-1}_1(\bar{f})\to S(\nu_i-1,1)\times \prod_{j\neq i}\GL_{\nu_j}$.
    Accordingly, it remains to determine $(f_{\tau})_{\tau \in E}$ for $(\bar f_{\tau})_{\tau \in E}$. 

    Let us fix 
    $((\bar\phi_i,\phi'_i),(\bar\phi_j)_{j\in I\setminus\{i\}})\in S(\nu_i-1,1)\times \prod_{j\neq i}\GL_{\nu_j}$ and fix a splitting $\Bbbk^{\nu_i} = \image\bar\phi_i \oplus W$, where $W$ is one dimensional subspace of $\Bbbk^{\nu_i}$. 
    Since the condition $\Bbbk^{\nu_i}=\image \bar\psi_i \oplus W$ for $\bar\psi_i:\Bbbk^{\nu_i-1}\rightarrow\Bbbk^{\nu_i}$ assigns an open property, 
    there is an open subset $U \subset S(\nu_{i}-1, 1) \times \prod_{j \in I \setminus \{i\}} \GL_{\nu_j}$ 
    such that 
    a direct sum decomposition 
    $\Bbbk^{\nu_i}=\image \bar\psi_i \oplus W$ 
    holds for any 
    $u=( (\bar\psi_i, \psi'_i), (\bar\psi_j)_{j \in I \setminus \{i\}}) \in U $. The projections $\Proj_1 : \Bbbk^{\nu_i} \to \image \bar\psi_i$ and $\Proj_2 :  \Bbbk^{\nu_i} \to W$ depend algebraically on $u \in U$.  
    Let us also fix an isomorphism $\gamma : W \cong \Bbbk$. 
    
    For $s(\tau) \neq i$, $f_{\tau}$ is determined by $\bar f_{\tau}$ and $(\bar\psi_j)_{j \in I}$ as $f_{\tau}=\bar\psi_{t(\tau)}\bar f_{\tau}\bar\psi^{-1}_{s(\tau)}$. 
    For $\tau_\ell=(i,\ell)\in E$ with $\ell\in\out_1(i)$, by using the isomorphism 
    $\Hom (\Bbbk^{\nu_i}, \Bbbk^{\nu_{\ell}}) \cong \Hom ( \image\bar\psi_i, \Bbbk^{\nu_{\ell}}) \oplus \Hom (W, \Bbbk^{\nu_{\ell}})$, 
    a map $f_{\tau_\ell}=(f_{1\tau_\ell}, f_{2\tau_\ell})$ is given by  
    \[
    f_{1\tau_\ell}(x)=\bar\psi_{\ell}\bar f_{\tau_\ell}\bar\psi^{-1}_{i}(\Proj_1(x)), \quad
    f_{2\tau_\ell}(x)=\gamma(\Proj_2(x))\bar\psi_\ell (v_\ell)
    \]
    for each  $u\in U$ and $(v_\ell)_{\ell\in\out_1(i)}\in \Exttilde_i(\bar f)$.     
    Therefore, we obtain a morphism $U\times \Exttilde_i(\bar f)\to p^{-1}(U)$. This morphism is an isomorphism because conversely any $f$ over $\bar f$ is decomposed as above. 
    Since each fiber is nonempty, $\varpi_1$ is surjective.  

    (2)  Let us define morphisms
    $p_1 : \varpi^{-1}_1(\bar\Lambda^{ 1}) \to 
    S(\nu_{i}-1, 1) \times \prod_{j \in I \setminus \{i\}} \GL_{\nu_j} \times \bar\Lambda^{ 1}$ 
    and 
    $p_2:S(\nu_{i}-1, 1) \times \prod_{j \in I \setminus \{i\}} \GL_{\nu_j} \times \bar\Lambda^{ 1}\to\bar\Lambda^{ 1}$
    by sending 
    $( (\bar\phi_j)_{j \in I}, (\phi'_j)_{j \in I}, (f_{\tau})_{\tau \in E})$ 
    to
    $( (\bar\phi_i, \phi'_i), (\bar\phi_j)_{j \in I \setminus \{i\}}, (\bar f_{\tau})_{\tau \in E})$ and the projection, respectively. 
    Then, we have a decomposition $\varpi_1|_{\varpi_1^{-1}(\bar\Lambda^{ 1})}=p_2 p_1$. 
    From Lemma \ref{lem:smooth}, we see that $p_2$ is smooth, so we show that $p_1$ is smooth.

    Let us fix $( (\bar\phi_i, \phi'_i), (\bar\phi_j)_{j \in I \setminus \{i\}}, (\bar f_{\tau})_{\tau \in E}) \in 
    S(\nu_{i}-1, 1) \times \prod_{j \in I \setminus \{i\}} \GL_{\nu_j} \times \bar\Lambda^{ 1}$, and a complementary subspace $W\subset \Bbbk^{\nu_i}$ to $\image \bar\phi_i$. 
    From the same argument in (1), 
    there is an open subset $U \subset S(\nu_{i}-1, 1) \times \prod_{j \in I \setminus \{i\}} \GL_{\nu_j} \times \bar\Lambda^{ 1}$ 
    such that  
    a direct sum decomposition 
    $\Bbbk^{\nu_i}=\image \bar\psi_i \oplus W$ 
    holds for any 
    $u=( (\bar\psi_i, \psi'_i), (\bar\psi_j)_{j \in I \setminus \{i\}}, (\bar g_\tau)_{\tau \in E}) \in U $. The projections $\Proj_1 : \Bbbk^{\nu_i} \to \image \bar\psi_i$ and $\Proj_2 :  \Bbbk^{\nu_i} \to W$ depend algebraically on $u \in U$.  
    We also fix an isomorphism $\gamma : W \cong \Bbbk$. 
    
    In the following, we construct an open set $U'\subset U$ and isomorphism $q: U'\times \Exttilde_i(\bar f)\rightarrow p_1^{-1}(U')$. 
    We note from (1) that the fiber 
    $p_1^{-1}  ( (\bar\psi_i, \psi'_i), (\bar\psi_j)_{j \in I \setminus \{i\}}, (\bar g_{\tau})_{\tau \in E})$
    is isomorphic to $\Exttilde_i(\bar g)$. 
    Let us fix a complementary subspace $X$ so that $\oplus_{\ell\in\out_1(i)}\Bbbk^{\nu_\ell}=\Exttilde_i(\bar f)\oplus X$. 
    We define an open subset $V\subset \bar\Lambda^{ 1}$ consisting of $\bar g= (\bar g_{\tau})_{\tau \in E}$ with $\oplus_{\ell\in\out_1(i)}\Bbbk^{\nu_\ell}=\Exttilde_i(\bar g) \oplus X$, and  
    set $U' = U \cap (S(\nu_{i}-1, 1) \times \prod_{j \in I \setminus \{i\}} \GL_{\nu_j} \times V)$. 
    We note that, for each $\bar g \in V$, we have an isomorphism $k_{\bar g} : \Exttilde_i(\bar f) \to \oplus_{\ell\in\out_1(i)}\Bbbk^{\nu_\ell} \to \Exttilde_i(\bar g)$ which depends algebraically on $\bar g$. 
    
    Then, the morphism $q :  U' \times \Exttilde_i(\bar f) \to p^{-1}_1(U')$ is defined similarly to the proof of (1) by sending 
    $( (\bar\psi_i, \psi'_i), (\bar\psi_j)_{j \in I \setminus \{i\}}, (\bar g_{\tau})_{\tau \in E}, v)$
    to 
    $( (\bar\psi_j)_{j\in I}, (\psi'_j)_{j \in I}, (g_{\tau})_{\tau \in E})$. 
    Namely, for $s(\tau)\neq i$, 
    $g_{\tau}: \Bbbk^{\nu_{s(\tau)}} \to \Bbbk^{\nu_{t(\tau)}}$ is given by
    $g_{\tau}(x) =\bar\psi_{t(\tau)}\bar g_\tau\bar\psi^{-1}_{s(\tau)}(x)$. 
    For $\tau_\ell=(i,\ell)\in E$ with $\ell\in\out_1(i)$,
    the map $g_{\tau_\ell}=(g_{1\tau_\ell}, g_{2\tau_\ell})$  is given by  
    \[
    g_{1\tau_\ell}(x)=\bar\psi_{\ell}\bar g_{\tau_\ell}\bar\psi^{-1}_{i}(\Proj_1(x)), \quad
    g_{2\tau_\ell}(x)=\gamma(\Proj_2(x))\bar{\psi}_\ell k_{\bar g}(v),
    \]
    where $v\in \Exttilde_i(\bar f)$. 
    From the construction, the morphism $q$ is an isomorphism.  
    Therefore, Lemma \ref{lem:smooth} shows the smoothness of $p_1$.    

    (3) Let $f = ( f_{\tau})_{\tau \in E} \in E_C(\nu)$. The fiber $\varpi^{-1}_2(f)$ consists of $( (\bar\phi_j)_{j \in I}, (\phi'_j)_{j \in I}, f)$ 
    satisfying the conditions in Definition \ref{defEprime}.
    Then $\bar\phi_i$ is injective and satisfies the condition $\image \bar\phi_i \supset \sum_{t(\tau) = i} \image f_{\tau} $, and $\phi'_i$ is surjective. 
    Such $(\bar\phi_j)_{j \in I}$ and $(\phi'_j)_{j \in I}$ exist if and only if $\nu_i-1 (=\dim \image \bar\phi_i ) \geq \dim \sum_{t(\tau) = i} \image f_{\tau}$. 
    This is equivalent to  $\nu_i - \dim \sum_{t(\tau) = i} \image f_{\tau} = \varepsilon_i (f) >0 $. 




    The map $\phi'_i$ corresponds bijectively to the non-zero linear map 
    $\tilde \phi'_i : \Bbbk^{\nu_i}/\sum_{t(\tau) = i} \image f_{\tau} \cong \Bbbk^{\varepsilon_i(f)} \to \Bbbk$, 
    thus it is parameterized by $(\Bbbk^{\varepsilon_i(f)})^* \setminus \{ 0 \}$. The map $\phi_j$ for $j \neq i$ is any automorphism of $\Bbbk^{\nu_j}$. Thus we have a surjection $\varpi^{-1}_2(f) \to \prod_{j \neq i} \GL_{\nu_j}\times \left((\Bbbk^{\varepsilon_i(f)})^* \setminus \{ 0 \} \right) =:B$. 
    It remains to prove that this is a principal $\GL_{\nu_{i-1}}$ bundle. 
    
    Let us consider the tautological morphism $\Phi'_i: \Bbbk^{\nu_i}\times B\to\Bbbk\times B$ associated with $\phi'_i$ on trivial vector bundles over $B$.
    Since $\Phi'_i$ is surjective,  $\ker \Phi'_i$ exists as a vector bundle (cf. Proposition 1.7.2 in \cite{Potier}). 
    Furthermore, since the dimensions of fibers of the vector bundles $\Bbbk^{\nu_i-1}\times B$ and $\ker\Phi'_i$ are the same, we have a principal $\GL_{\nu_i-1}$ bundle $\text{Isom}(\Bbbk^{\nu_i-1}\times B,\ker\Phi'_i)\subset \Hom(\Bbbk^{\nu_i-1}\times B,\ker\Phi'_i)$. 
    Since the map $\bar\phi_i$ is given by an isomorphism $ \Bbbk^{\nu_i-1} \cong \ker \phi'_i$, we see that the fiber $\varpi_2^{-1}(f)$ is isomorphic to $\text{Isom}(\Bbbk^{\nu_i-1}\times B,\ker\Phi'_i)$. 
    

    (4) Let $\mathcal{V}_j := \Bbbk^{\nu_j} \times \Lambda^{ 2}$ be a trivial vector bundle over $\Lambda^{ 2}$ and 
    $F_{\tau} : \mathcal{V}_{s(\tau)} \to \mathcal{V}_{t(\tau)} $ 
     be the tautological morphism. 
    Since $(F_{\tau})_{t(\tau)=i}: \oplus_{t(\tau)=i} \mathcal{V}_{s(\tau)} \to \mathcal{V}_i$ has the constant rank on $\Lambda^{ 2}$, 
    $\mathcal{C}  := \coker (F_{\tau})_{t(\tau)=i}$ is a vector bundle on $\Lambda^{ 2}$. 
    Let $(\mathcal{C}^{*})^\circ$ be the complement of the 0-section in the dual $\mathcal{C}^*$ of $\mathcal{C}$, 
    $\mathcal{W} := \Bbbk^{\nu_i} \times \prod_{j \in I \setminus \{i\}} GL_{\nu_j} \times (\mathcal{C}^*)^{\circ}$
    be the trivial bundle, 
    and  
    $\Phi'_i: \mathcal{W} \to 
    \Bbbk \times \prod_{j \in I \setminus \{i\}} \GL_{\nu_j} \times (\mathcal{C}^*)^{\circ}$
    be the tautological morphism associated with $\phi'_i\in (\mathcal{C}^*)^\circ$. 
    Then, $\Phi'_i$ becomes surjective, and hence we have a vector bundle $\mathcal{K}:= \ker \Phi'_i$ over $\prod_{j \in I \setminus \{i\}} \GL_{\nu_j} \times (\mathcal{C}^*)^{\circ}$. 

    Let $\overline{\mathcal{W}} := \Bbbk^{\nu_i-1} \times \prod_{j \in I \setminus \{i\}} \GL_{\nu_j} \times (\mathcal{C}^*)^{\circ}$. 
    From the similar argument to (3), we have a fiber bundle $\textrm{Isom}(\overline{\mathcal{W}},\mathcal{K})\subset \Hom(\overline{\mathcal{W}},\mathcal{K})$. 
    Let us define a morphism
     $\alpha: \varpi^{-1}_2(\Lambda^{ 2}) \to \textrm{Isom}(\overline{ \mathcal{W}}, \mathcal{K} ) $
    by sending $( (\bar\phi_j)_{j \in I}, (\phi'_j)_{j \in I}, (f_{\tau})_{\tau \in E}) \mapsto ( \bar\phi_i, (\bar\phi_j)_{j \in I \setminus \{i\}},  (\phi'_i, (f_{\tau})_{\tau \in E}))$. 
    Then as in (3), we see that $\alpha$ is an isomorphism. 

    The composition of $\alpha$ and the projection to $\Lambda^{ 2}$ is the same as the restriction of $\varpi_2$. Since $\textrm{Isom}(\overline{ \mathcal{W}}, \mathcal{K} )$ is a principal $\GL_{\nu_{i-1}}$ bundle over $\prod_{j \in I \setminus \{i\}} \GL_{\nu_j} \times (\mathcal{C}^*)^{\circ}$, $\varpi_2$ is smooth over $\Lambda^{ 2}$ from Lemma \ref{lem:smooth}. 
\end{proof}

\begin{rem}    
    The vector space $\Exttilde_i(f)$ is different from that appearing in the paper \cite{KS}. The distinction arises because the preprojective relation (\ref{eq:preprojective}) is generated by paths whose source and target vertices are identical, whereas the  commutative relation is generated by paths with distinct source and target vertices.
\end{rem}

\begin{lem}\label{lem:pullback_irr}
Let $f: X\rightarrow Y$ be a smooth morphism with connected fibers on algebraic varieties and $Y$ be irreducible. Then, $X$ is also irreducible. 
\end{lem}
\begin{proof}
Let us assume the irreducible decomposition $X=\bigcup_{k\in K}X_k$ with $|K|> 1$ and set disjoint open sets $U_1=X_1\setminus \bigcup_{k\in K, k\neq 1}X_k$ and $U_2=X_2\setminus \bigcup_{k\in K, k\neq 2}X_k$. Since $f$ is smooth, $f(U_1)$ and $f(U_2)$ are open in $Y$. Then, the irreducibility of $Y$ implies $f(U_1)\cap f(U_2)\neq \emptyset$, and we take a point $y$ from this intersection. From the assumption, $f^{-1}(y)$ is smooth and connected, hence irreducible. Therefore the nonempty open sets $f^{-1}(y)\cap U_1$ and $f^{-1}(y)\cap U_2$ have an intersection $f^{-1}(y)\cap U_1\cap U_2\neq \emptyset$. This contradicts to the disjointness of $U_1$ and $U_2$. 
\end{proof}

By Proposition \ref{prop:varpi} (2) and (4) and Lemma \ref{lem:pullback_irr}, the algebraic sets $\overline{\varpi_2 \varpi^{-1}_1(\bar\Lambda^{ 1})}$ and $\overline{\varpi_1 \varpi^{-1}_2(\Lambda^{ 2})}$ are irreducible.
Using the notations in Proposition \ref{prop:varpi}, we now define the Kashiwara operators in the following way. 

\begin{defn}\label{defn:ef}
(1)
For $\bar \Lambda \in B_C(\nu +\alpha_i)$, we define $\tilde f_i \bar \Lambda \in B_C(\nu) \sqcup \{ 0 \}$ by 

\[
\tilde f_i \bar \Lambda
:=
\left\{
\begin{matrix}
\overline{\varpi_2 \varpi^{-1}_1(\bar \Lambda^{ 1})} 
&
\text{if $\overline{\varpi_2 \varpi^{-1}_1(\bar \Lambda^{ 1})} \in B_C(\nu)$},
\\
0
&
\text{otherwise}.
\end{matrix}
\right.
\] 

(2)
For $\Lambda \in B_C(\nu)$, we define $\tilde e_i \Lambda \in B_C(\nu + \alpha_i) \sqcup \{ 0 \}$ by 
\[
\tilde e_i \Lambda
:=
\left\{
\begin{matrix}
\overline{\varpi_1 \varpi^{-1}_2(\Lambda^{ 2})} 
&
\text{if $\varepsilon_i(\Lambda)>0$ and $\overline{\varpi_1 \varpi^{-1}_2(\Lambda^{ 2})} \in B_C(\nu+\alpha_i)$,} 
\\
0
&
\text{otherwise}.
\end{matrix}
\right.
\]    
\end{defn}

Note that even though $\overline{\varpi_2 \varpi^{-1}_1(\bar\Lambda^{ 1})}$ and $\overline{\varpi_1 \varpi^{-1}_2(\Lambda^{ 2})}$ are irreducible, 
they are not necessarily irreducible components of $E_C(\nu)$ and $E_C(\nu+\alpha_i)$, respectively. 
We find such examples in the $2 \times 2$ case later (see Example \ref{ex:1112}).

\begin{lem}\label{prop:ef} \hfill
    \begin{enumerate}
        \item If $ \tilde e_i \Lambda = \bar \Lambda \in B_C(\nu + \alpha_i)$, we have $\Lambda = \tilde f_i \bar \Lambda$.

        \item If $ \tilde f_i \bar \Lambda= \Lambda \in B_C(\nu)$, we have  $ \bar \Lambda = \tilde e_i \Lambda $. 
    \end{enumerate}
\end{lem}

\begin{proof}
    (1) 
    It follows from the assumption $\bar \Lambda = \overline{\varpi_1 \varpi^{-1}_2(\Lambda^{2})}$ that we have 
    $\varpi_2^{-1} \Lambda^{2} \subset \varpi_1^{-1} \bar \Lambda$. 
    Since $\bar \Lambda^{ 1}$ is open in $\bar \Lambda$, 
    $\varpi_2^{-1} \Lambda^{2} \cap \varpi_1^{-1} \bar \Lambda^{1}$ 
    is also open in 
    $\varpi_2^{-1} \Lambda^{2} \cap \varpi_1^{-1} \bar \Lambda = \varpi_2^{-1} \Lambda^{2}$.   
    We have 
    $\varpi_1(\varpi_2^{-1} \Lambda^{2} \cap \varpi_1^{-1} \bar \Lambda^{1}) = \varpi_1(\varpi_2^{-1}\Lambda^{2})\cap \bar\Lambda^{1}$
    by the set-theoretic projection formula, which is proved by an elementwise argument. 
    If $\varpi_2^{-1} \Lambda^{2} \cap \varpi_1^{-1} \bar \Lambda^{1}=\emptyset$, then we have $\varpi_1(\varpi_2^{-1}\Lambda^{ 2})\cap \bar\Lambda^{1}=\varpi_1(\varpi_2^{-1} \Lambda^{2} \cap \varpi_1^{-1} \bar \Lambda^{1})=\emptyset$. This implies that $\varpi_1\varpi_2^{-1}\Lambda^{2}$ is a subset of $\bar\Lambda\setminus\bar\Lambda^{ 1}$, which is a proper closed set of $\bar\Lambda$. This contradicts to the assumption $\bar\Lambda=\overline{\varpi_1 \varpi^{-1}_2(\Lambda^{ 2})} \in B_C(\nu + \alpha_i)$.
    This implies that $\varpi_2^{-1} \Lambda^{2} \cap \varpi_1^{-1} \bar \Lambda^{1}$ is a non-empty open subset of $\varpi_2^{-1} \Lambda^{2}$.
      
    From Proposition \ref{prop:varpi} (4), $\varpi_2|_{\varpi_2^{-1} \Lambda^{2}}: \varpi_2^{-1} \Lambda^{2} \to \Lambda^{ 2}$ is smooth and hence an open morphism. 
    Thus $\varpi_2(\varpi_2^{-1}\Lambda^{ 2}\cap \varpi_1^{-1}\bar\Lambda^{ 1})$ is a nonempty open subset of $\Lambda$. 
    Then, from the irreducibility of $\Lambda$, we have $\Lambda=\overline{\varpi_2(\varpi_2^{-1}\Lambda^{ 2}\cap \varpi_1^{-1}\bar\Lambda^{ 1})}$, which implies $\Lambda\subset \overline{\varpi_2\varpi_1^{-1}\bar\Lambda^{ 1}}$. 
    Since $\Lambda$ is an irreducible component of $E_C(\nu)$ and $\overline{\varpi_2\varpi_1^{-1}\bar\Lambda^{ 1}}$ is an irreducible closed subset of $E_C(\nu)$, we have $\Lambda=\overline{\varpi_2\varpi_1^{-1}\bar\Lambda^{ 1}}$.

    (2) $\Lambda =\tilde f_i \bar \Lambda$ implies $\varepsilon_i (\Lambda) > 0$. By Proposition \ref{prop:varpi} (3), we see that $\varpi_2|_{\varpi_2^{-1} \Lambda} : \varpi^{-1}_2 \Lambda \to \Lambda$ is surjective. 
    With this, we can prove (2) in the same way as (1). 
\end{proof}

With Definition \ref{defn:modulidef}, \ref{defn:epsilonphi} and \ref{defn:ef}, we can now state and prove the main result of this paper. 

\begin{thm}\label{thm:BC}
    The tuple $(B_C, \weight, (\varepsilon_i)_{i \in I}, (\varphi_i)_{i \in I}, (\tilde e_i)_{i \in I}, (\tilde f_i)_{i \in I})$ is a Kashiwara crystal with respect to the root system associated with the quiver $\overrightarrow{G}_{m}$. 
\end{thm}
\begin{proof}
    We check the defining properties of the Kashiwara crystal in Definition \ref{defn:crystal}. 
    
    The property (1) is satisfied by the definition of the function $\varphi_i$. 

    Assume $\Lambda \in B_C(\nu)$ and $\tilde e_i \Lambda \neq 0$. Then  $\weight(\tilde e_i \Lambda) =\nu+\alpha_i$ follows from Definition \ref{defn:ef}. From the proof of Proposition \ref{prop:varpi} (1), we see that for a representation $(f_{\tau})_{\tau \in E} \in \Lambda$ and a subrepresentation $(\bar f_{\tau})_{\tau \in E} \in \tilde e_i \Lambda$ the ranks of 
    $(f_{\tau})_{\tau \in E, t(\tau) =i}$ 
    and  
    $(\bar f_{\tau})_{\tau \in E, t(\tau) =i}$ coincide. 
    Then, by setting $r$ to be its rank, we have $\varepsilon_i((f_{\tau})_{\tau \in E})= \nu_{i}-r$ and $\varepsilon_i((\bar f_{\tau})_{\tau \in E})= \nu_{i}-1-r$, which implies $\varepsilon_i(\Lambda) - 1 = \varepsilon_i(\tilde e_i \Lambda)$. The equality $\varphi_i(\tilde e_i \Lambda) = \varphi_i(\Lambda) + 1 $ follows from the definition of $\varphi_i$ and $\langle h_i , \alpha_i \rangle = 2$. These prove the property (2). 

    The property (3) is proved similarly. 
    The property (4) is proved in Lemma \ref{prop:ef}. 
    The property (5) is satisfied because no elements in $B_C$ take the value $-\infty$ by $ \varphi_i$. 
\end{proof}

In the remaining of this section,  we construct the $\ast$-Kashiwara operators, which also give a crystal structure on $B_C$. 
Since the proofs of the statements below are similar to the originals, we omit them. 
We note that the $\ast$-Kashiwara operators may be defined taking dual of representations and using $\overrightarrow{G}^{}_{m} \cong \overrightarrow{G}^{op}_{m}$. This construction is explained at the end of this section.

For a fixed $i \in I$ and $\nu +\alpha_i \in -{\rm Q}_+$, let us consider 
the following correspondence 
\[
E_C(\nu+\alpha_i) \xleftarrow{\varpi_3} E'_C(-\alpha_i, \nu+\alpha_i) \xrightarrow{\varpi_4} E_C(\nu).
\]
The morphism $\varpi_3 : E'_C(-\alpha_i, \nu+\alpha_i) \to E_C(\nu+\alpha_i) $ is defined by 
$((\bar\phi_j)_{j \in I}, (\phi'_j)_{j \in I}, (f_{\tau})_{\tau \in E} )
\mapsto (f'_{\tau})_{\tau \in E}$, 
where $(f'_{\tau})_{\tau \in E}$ is the quotient representation of $(f_{\tau})_{\tau \in E}$
and 
$\varpi_4 : E'_C(-\alpha_i, \nu+\alpha_i) \to E_C(\nu) $ by 
$((\bar\phi_j)_{j \in I}, (\phi'_j)_{j \in I}, (f_{\tau})_{\tau \in E} ) \mapsto (f_{\tau})_{\tau \in E}$.

\begin{defn}\hfill
    \begin{enumerate}
    \item For $f \in E_C(\nu)$, $\varepsilon^*_i(f) := 
    \dim \kernel ( (f_{\tau})_{\tau \in E, s(\tau) =i}:  
    \Bbbk^{\nu_i} \to \bigoplus_{\tau \in E, s(\tau) =i} \Bbbk^{\nu_{t(\tau)}})$. 
    \item For $\Lambda \in B_C(\nu)$, $\varepsilon^*_i(\Lambda) := \min_{f \in \Lambda} \{ \varepsilon_i^*(f) \}$.
    \item For $\Lambda \in B_C(\nu)$, $\varphi^*_i(\Lambda) := \varepsilon^*_i(\Lambda) + \langle h_i, \weight(\Lambda) \rangle$.
    \end{enumerate}
\end{defn}

\begin{defn}\label{defFibstar}
\hfill
\begin{enumerate}
\item $\inward_1(i):=\{j\in I~|~(j,i)\in E\}$.
\item $\inward_2(i):=\{k\in I~|~(k,j_1),(k,j_2)\in E~\text{for $j_1,j_2\in \inward_1(i)$ and $j_1\neq j_2$}\}$.
\item For $j_1,j_2\in \inward_1(i)$ with $j_1<_\text{lex}j_2$, a vertex $k(j_1,j_2)\in\inward_2(i)$ is defined by the unique vertex satisfying $(k(j_1,j_2),j_1), (k(j_1,j_2),j_2)\in E$. Note that any $k\in \inward_2(i)$ is uniquely expressed as $k(j_1,j_2)$ by $j_1,j_2\in\inward_1(i)$.

\item A linear map $F_i(f):\oplus_{j\in\inward_1(i)}\Bbbk^{\nu_j}\rightarrow\oplus_{k\in \inward_2(i)}\Bbbk^{\nu_k}$ 
is defined for each direct summand 
$F_i(f)_{jk}: \Bbbk^{\nu_j} \to \Bbbk^{\nu_k} $          
by 
\[F_i(f)_{jk}
=
\begin{cases}
f_{(k(j,j'), j)}^T & \text{if there is $j' \in \inward_1(i)$ with $j<_{\text{lex}}j'$ and $k(j,j') = k$},\\
-f_{(k(j',j), j)}^T & \text{if there is $j' \in \inward_1(i)$ with $j'<_{\text{lex}}j$ and $k(j',j) = k$},\\
0 & \text{otherwise}.
\end{cases}
\]
Here, in the case of $|\inward_1(i)|=1$, the linear map is understood as $F_i(f)=0$.

\item $\Exttilde^*_i(f):=(\kernel F_i(f))^*$.
\end{enumerate}
\end{defn}

\begin{prop}\label{prop:pi} \hfill
    \begin{enumerate}
    \item The morphism $\varpi_3$ is surjective and $\varpi_3^{-1}(f')$ is a vector bundle over $S(1,\nu_i-1) \times \prod_{j \neq i} \GL_{\nu_j}$ whose fibers are isomorphic to $ \Exttilde^*_i(f') $. 
    
    \item Let $\Lambda' \in B_C(\nu+\alpha_i)$. Define $\Lambda^{'3}$ to be the open subset of $\Lambda'$ consisting of points $f'$ at which the function 
    $f' \to \dim \Exttilde^*_i(f')$ takes the minimum value in $\Lambda'$. 
    Then the restriction of $\varpi_3$ to $\varpi^{-1}_3(\Lambda^{'3})$ is a smooth surjective morphism to $\Lambda^{'3}$ with connected fibers. 
    
    \item The image of the morphism $\varpi_4$ is the closed subset of $E_C(\nu)$ consisting of points $f$ with $\varepsilon^*_i(f)>0$.
    For $f \in E_C(\nu)$ with  $\varepsilon^*_i(f)>0$,  the fiber $\varpi^{-1}_4(f)$ is a fiber bundle 
    over $\prod_{j \neq i} \GL_{\nu_j} \times \left(
    \Bbbk^{\varepsilon^*_i(f)} \setminus \{ 0 \} \right)$ 
    whose fibers are isomorphic to $\GL_{\nu_i-1}$. 
    
    \item Assume that $\Lambda \in B_C(\nu)$ satisfies $\varepsilon^*_i(\Lambda)>0$. 
    Define $\Lambda^{ 4}$ to be the open subset of $\Lambda$ consisting of points at which the function $\varepsilon_i^*$ takes the minimum value in $\Lambda$.
    Then the restriction of $\varpi_4$ to $\varpi^{-1}_4(\Lambda^{ 4})$ is a smooth surjective morphism to $\Lambda^{ 4}$ with connected fibers. 
    \end{enumerate}
\end{prop}


\begin{defn}\label{def:efstar}
(1)
For $\Lambda' \in B_C(\nu +\alpha_i)$, we define $\tilde f_i^* \Lambda' \in B_C(\nu) \sqcup \{ 0 \}$ by 

\[
\tilde f_i^* \Lambda'
:=
\left\{
\begin{matrix}
\overline{\varpi_4 \varpi^{-1}_3(\Lambda^{'3})} 
&
\text{if $\overline{\varpi_4 \varpi^{-1}_3(\Lambda^{'3})} \in B_C(\nu)$},
\\
0
&
\text{otherwise}.
\end{matrix}
\right.
\] 

(2)
For $\Lambda \in B_C(\nu)$, we define $\tilde e_i^* \Lambda \in B_C(\nu + \alpha_i) \sqcup \{ 0 \}$ by 
\[
\tilde e_i^* \Lambda
:=
\left\{
\begin{matrix}
\overline{\varpi_3 \varpi^{-1}_4(\Lambda^{ 4})} 
&
\text{if $\varepsilon^*_i(\Lambda)>0$ and $\overline{\varpi_3 \varpi^{-1}_4(\Lambda^{ 4})} \in B_C(\nu+\alpha_i)$ },
\\
0
&
\text{otherwise}.
\end{matrix}
\right.
\] 
\end{defn}

\begin{thm}\label{thm:BCstar}
    The tuple $(B_C, \weight, (\varepsilon^*_i)_{i \in I}, (\varphi^*_i)_{i \in I}, (\tilde e^*_i)_{i \in I}, (\tilde f^*_i)_{i \in I})$ 
    is a crystal with respect to the root system associated with the quiver $\overrightarrow{G}_{m}$. 
\end{thm}

We denote by $\overrightarrow{G}^{op}_{m}$ the quiver with the same vertices as $\overrightarrow{G}_{m}$ but reversed arrows. 
We define $E^{op}_C(\nu)$ and $E'^{op}_C(\bar\nu,  \nu')$ for $\overrightarrow{G}^{op}_{m}$
in the similar way to $E^{}_C(\nu)$ and $E'^{}_C(\bar\nu, \nu')$. 
Let $a$ be the isomorphism of quivers $\overrightarrow{G}^{}_{m} \to \overrightarrow{G}^{op}_{m}$ which send the vertex $(i_1,\dots,i_d)$ to $(m_1-i_1+1,\dots,m_d-i_d+1)$. 
Note that on the vertex set $I$ the map $a$ is an involution. 
Let $a^*\nu$  be the weight given by $a^*\nu = -\sum_{i \in I} \nu_{a(i)} \alpha_i$ for $\nu = -\sum_{i \in I} \nu_{i} \alpha_i$. 
We have morphisms of algebraic varieties $D: E_C(\nu) \to E^{op}_C(\nu)$ and $D^{op}: E^{op}_C(\nu) \to E^{}_C(\nu)$ given by taking dual of linear maps and identifications of vector spaces $\Bbbk^{\nu_i} \cong (\Bbbk^{\nu_i})^*$ given by the standard inner product on $\Bbbk^{\nu_i}$. Under this identification $D$ sends $(f_{\tau})$ to $(f_{\tau}^{T})$. 
These morphisms satisfy identities $D^{op} \circ D^{} = {\rm id}$ and $D^{} \circ D^{op} = {\rm id}$, which imply that $D$ and $D^{op}$ are isomorphisms. 
We have an isomorphism of varieties $a^*: E^{op}_C(\nu) \to E^{}_C(a^*\nu)$, which is a rearrangement of representation induced by the permutation of vertices.  
Since taking dual sends the exact sequence 
$0 \to (\bar f_{\tau}) \to (f_{\tau}) \to S_i \to 0$ to $0 \to S_i \to (f_{\tau})^* \to (\bar f_{\tau})^* \to 0$, we also have isomorphisms $D: E'^{}_C(\bar \nu, \nu') \to E'^{op}_C(\nu', \bar \nu)$, $D^{op}: E'^{op}_C(\bar \nu, \nu') \to E'^{}_C(\nu', \bar \nu)$ and $a^*: E'^{op}_C(\bar \nu, \nu') \to E'^{}_C(a^*\bar \nu, a^*\nu')$. 
Then, we obtain the following commutative diagram 
\begin{center}
\begin{tikzcd}
    E_C(\nu) \arrow[d, "a^{*} \circ D"] & E'_C(\nu, -\alpha_i) \arrow[l, "\varpi_1"'] \arrow[r, "\varpi_2"] \arrow[d, "a^{*} \circ D"]& E_C(\nu-\alpha_i) \arrow[d, "a^{*} \circ D"]\\
    E_C(a^*\nu)  & E'_C(-\alpha_{a(i)}, a^*\nu) \arrow[l, "\varpi_3"'] \arrow[r, "\varpi_4"]& E_C(a^*\nu-\alpha_{a(i)}),
\end{tikzcd}
\end{center}
where all vertical morphisms are isomorphisms. From this, we may deduce 
Proposition \ref{prop:pi}, Theorem \ref{thm:BCstar}, Theorem \ref{thm:ef_an} (2), Theorem \ref{thm:2x2efstar} and Proposition \ref{prop:varepsilonstar2x2} by Proposition \ref{prop:varpi}, Theorem \ref{thm:BC}, Theorem \ref{thm:ef_an} (1), Theorem \ref{thm:2x2ef} and Proposition \ref{prop:varepsilon2x2}, respectively. 
In particular, we have a bijection $D_{a}:B_C(\nu) \to B_C(a^*\nu)$ induced from $a^{*} \circ D$ and obtain the following. 
\begin{thm}\label{thm:antipode}
For $\Lambda \in B_C(\nu)$, we have 
\begin{gather*}
\tilde f_i^* \Lambda:=
\left\{
\begin{matrix}
(D_{a})^{-1} \circ \tilde f_{a(i)} \circ D_{a}(\Lambda) & \text{if $\tilde f_{a(i)} \circ D_{a}(\Lambda)\neq 0$}, \\
0&\text{otherwise},
\end{matrix}
\right. \\
\tilde e_i^* \Lambda :=
\left\{
\begin{matrix}
(D_{a})^{-1} \circ \tilde e_{a(i)} \circ D_{a}(\Lambda) & \text{if $\tilde e_{a(i)} \circ D_{a}(\Lambda)\neq 0$ }, \\
0 & \text{otherwise}.
\end{matrix}
\right.
\end{gather*}
\end{thm}

\begin{rem}
    In the construction of Kashiwara and Saito \cite{KS}, the permutation of vertices is not required because their quiver is the double of a quiver. 
\end{rem}

\section{Kashiwara crystals in the 1-parameter case}\label{sect:An}
In this section, we explicitly describe the Kashiwara operators in Definition \ref{defn:ef} and \ref{def:efstar} for $\overrightarrow{G}_{(n)}$ (i.e., $A_n$ quiver with the equi-orientation): 
\begin{center}
\begin{tikzcd}
1 \arrow[r,] & 2 \arrow[r,] & \cdots  \arrow[r,] & n.
\end{tikzcd}
\end{center}
This setting corresponds to studying one-parameter persistence modules. 
For a representation $f$, the linear map on an arrow $i\rightarrow i+1$ is written as $f_i: \Bbbk^{\nu_{i}}\to\Bbbk^{\nu_{i+1}}$. Since we do not have any relations in the quiver, $E_C(\nu)$ simply becomes a vector space specified by the dimension vector $(\nu_i)_{i=1}^n\in \N^n$, hence $B_C(\nu)=\{E_C(\nu)$\}. 

\begin{thm}\label{thm:ef_an}
\begin{enumerate}
\item 
The action of the Kashiwara operators is described as follows.
\begin{align*}
&\tilde e_i E_C(\nu) =
        \begin{cases}
            E_C(\nu+\alpha_i)\quad\quad\!\! &  \text{if $\nu_{i-1} < \nu_i$}, \\
            0 & \text{if $\nu_{i-1}\geq\nu_i$}
        \end{cases}\\
&\tilde f_i E_C(\nu) =
        \begin{cases}
            E_C(\nu-\alpha_i)\quad\quad\!\! &  \text{if $\nu_{i-1}\leq \nu_i$}, \\
            0 & \text{if $\nu_{i-1}>\nu_i$},
        \end{cases}
\end{align*}
where we set $\nu_{-1}=0$.

\item
The action of the $\ast$-Kashiwara operators is described as follows.
\begin{align*}
&\tilde e^{*}_i E_C(\nu) =
    \begin{cases}
        E_C(\nu+\alpha_i)\quad\quad\!\! &  \text{if $\nu_{i} > \nu_{i+1}$}, \\
        0 & \text{if $\nu_{i}\leq\nu_{i+1}$}
    \end{cases}\\
&\tilde f^{*}_i E_C(\nu) =
    \begin{cases}
        E_C(\nu-\alpha_i)\quad\quad\!\! &  \text{if $\nu_{i}\geq \nu_i$}, \\
        0 & \text{if $\nu_{i}<\nu_{i+1}$},
    \end{cases}
\end{align*}
where we set $\nu_{n+1}=0$.

\end{enumerate}
\end{thm}
\begin{proof}
We first consider the action of $\tilde e_i$. For a given $f\in \Lambda^2\subset E_C(\nu)$, we study whether we can construct the following commutative diagram
\[
\begin{tikzcd}
    & \vdots\arrow [d]& \vdots\arrow[d]& \vdots\arrow[d]& \\
    0 \arrow[r]& \Bbbk^{\nu_{i-1}} \arrow[r] \arrow[d,"\bar{f}_{i-1}"]& \Bbbk^{\nu_{i-1}} \arrow[r]\arrow[d,"f_{i-1}"] & 0 \arrow[r]\arrow[d] & 0 \\
   0 \arrow[r]& \Bbbk^{\nu_i-1} \arrow[r,"\bar\phi_i"]\arrow[d] & \Bbbk^{\nu_i} \arrow[r,"\phi'_i"]\arrow[d] & \Bbbk \arrow[r]\arrow[d] & 0\\
    & \vdots & \vdots& \vdots& 
%
\end{tikzcd}
\]
with the horizontal sequences to be exact. Since $f\in \Lambda^2$ implies $\rank f_{i-1}=\min \{\nu_{i-1},\nu_i\}$, $\varepsilon_i(f)>0$ is equivalent to $\nu_{i-1}<\nu_i$. Therefore, we obtain
the statement of $\tilde e_i$ from Definition \ref{defn:ef}. 

From this formula, we have 
\[
\tilde e_i E_C(\nu-\alpha_i)=\begin{cases}
E_C(\nu) & \text{if $\nu_{i-1}<\nu_i+1$},\\
0 & \text{if $\nu_{i-1}\geq \nu_i+1$}, 
\end{cases}
\]
which leads to the statement of $\tilde f_i$.

The statement for $\tilde e^{*}_{i}$ and $\tilde f^{*}_{i}$ is proved similarly.
\end{proof}

As a corollary, we obtain some properties of $B_C$ for one-parameter persistence. 
The crystal is called \emph{upper seminormal} if 
$\varepsilon_i(\Lambda) = \varepsilon_i'(\Lambda):= \max \{ k \in \mathbb{N} \ \vert \ \tilde e^{k}_i \Lambda \neq 0 \}$
holds for any $i$ and $\Lambda$.

\begin{cor}
\begin{enumerate}
    \item For $\overrightarrow{G}_{(1)}$, $B_C$ is isomorphic to $B(\infty)$. 
    \item For $\overrightarrow{G}_{(n)}$, the crystal $B_C$ is connected. 
    \item For $\overrightarrow{G}_{(n)}$, there exists a embedding of crystals $B_C \to B(\infty)$. In particular, the crystal $B_C$ is connected. 
    \item For $\overrightarrow{G}_{(n)}$, the crystal $B_C$ is upper seminormal. 
\end{enumerate}
\end{cor}

\begin{proof}
\begin{enumerate}
\item 
This follows from the description of $B(\infty)$ of type $A_1$ (\cite{NZ} \S 5).
\item
Let $u_C\in B_C(0)$ ($u_\infty\in B(\infty,0)$, respectively) be the unique irreducible component with the zero dimension vector. 
By Theorem \ref{thm:ef_an} (1), we have $E_C(\nu)=\tilde f^{\nu_1}_{1}  \cdots\tilde f^{\nu_n}_{n} u_C$. This implies that $B_C$ is connected. 
\item
Let $g:B_C \to B(\infty)$ be a map defined by $g(E_C(\nu))=\tilde f^{\nu_1}_{1} \cdots\tilde f^{\nu_n}_{n} u_\infty$. 
It can be checked that this gives a map of crystals, for example using the polyhedral realization of Nakashima and Zelevinsky \cite[\S 6]{NZ}. 
\item 
We have $\varepsilon_i(E_C(\nu))=\dim \Bbbk^{\nu_i}-\max\{\rank f_{i-1}\}= \nu_i-\min\{\nu_{i-1}, \nu_i\}$ by the definition of $\varepsilon_i$. 
We have $\varepsilon'_i(E_C(\nu))=\max\{0,\nu_i-\nu_{i-1}\}$ by Theorem \ref{thm:ef_an} (1). 
These imply the claim. 
\end{enumerate}
\end{proof}

\section{Kashiwara crystals in the 2 $\times$ 2 case}\label{sect:2x2}

In this section, we explicitly describe the Kashiwara operators in Definition \ref{defn:ef} and \ref{def:efstar} in the $2 \times 2$ case. We note that the generalized Cartan matrix associated to $\overrightarrow{G}_{(2,2)}$ is the same as that of the affine root system $\tilde A_3$. 

Let us enumerate the vertices of $\overrightarrow{G}_{(2,2)}$ by 1,2,3,4 as in the diagram below. The dimension vector and its corresponding weight are written by $(\nu_1, \nu_2, \nu_3, \nu_4) \in \mathbb{N}$ and $\nu=-\sum_{i=1}^4\nu_i\alpha_i\in P$, respectively. 
For a representation $f$ and an arrow $i \to j$ in $\overrightarrow{G}_{(2,2)}$, the linear map 
$\Bbbk^{\nu_i} \to \Bbbk^{\nu_j}$ is written as $f_{ij}$. 

\begin{center}
\begin{tikzcd}
    3 \arrow[r, ]& 4 \\
    1 \arrow[r, ] \arrow[u, ]& 2 \arrow[u, ]
\end{tikzcd},
\hspace{10pt}
\begin{tikzcd}
    \Bbbk^{\nu_3} \arrow[r, "f_{34}"] & \Bbbk^{\nu_4} \\
    \Bbbk^{\nu_1} \arrow[r, "f_{12}"'] \arrow[u, "f_{13}"]& \Bbbk^{\nu_2} \arrow[u, "f_{24}"']
\end{tikzcd}.
\end{center}

\begin{prop}\label{prop:irr2x2}\hfill
    \begin{enumerate}
        \item For $\nu_1+\nu_4 \geq \nu_2+\nu_3$, 
        there is a bijection between 
        $B_C(\nu)$ and the set 
        $\{ (r_1, r_2) \in \mathbb N \times \mathbb N ~\vert~
        r_1 + r_2 = \nu_2 + \nu_3
        , 0 \leq r_1 \leq \nu_1
        , 0 \leq r_2 \leq \nu_4
        \}$. 
        
        \item For $\nu_1+\nu_4 < \nu_2+\nu_3$, 
         $\# B_C(\nu) =1 $. 
    \end{enumerate}
\end{prop}

\begin{proof}
    Let $W(a_1, a_2, a_3)$ be the subvariety of $ \Hom(\Bbbk^{a_1} , \Bbbk^{a_2}) \times \Hom(\Bbbk^{a_2} , \Bbbk^{a_3})$ consisting of $(\alpha_1, \alpha_2) \in \Hom(\Bbbk^{a_1} , \Bbbk^{a_2}) \times \Hom(\Bbbk^{a_2} , \Bbbk^{a_3})$ satisfying $\alpha_2 \circ \alpha_1 = 0$. 
    Let $p: E_C(\nu) \to W(\nu_1, \nu_2+\nu_3, \nu_4)$ be a morphism sending $(f_{\tau})$ to 
    $( 
    \begin{bmatrix}
    f_{12} & f_{13}
    \end{bmatrix}^T, 
    \begin{bmatrix}
    f_{24} & -f_{34}
    \end{bmatrix}
    )$. 
    The morphism $p$ is an isomorphism because the commutativity relation 
    $f_{24} \circ f_{12}= f_{34} \circ f_{13}$
    corresponds to the  relation
    $\alpha_2 \circ \alpha_1 = 0$.

    The changes of bases on $\Bbbk^{a_i}$ ($i=1,2,3$) induce the group action of $\GL_{a_1} \times \GL_{a_2} \times \GL_{a_3}$ on $W(a_1, a_2, a_3)$. 
    The orbits under this action are parameterized by $r_1 = \mathrm{rank} (\alpha_1)$ and $r_2 = \mathrm{rank} (\alpha_2)$, where 
    $(r_1, r_2)$ satisfies $ 0 \leq r_1 \leq a_1$, $r_1+r_2 \leq a_2$ and $0 \leq r_2 \leq a_3$. 
    Let us write the orbit corresponding to $(r_1, r_2)$ by $\mathbb{O}(r_1, r_2)$. 

    We claim that $\overline{\mathbb{O}(r_1, r_2)} \supset \mathbb{O}(r'_1, r'_2)$ if and only if $r_1 \geq r'_1$ and $r_2 \geq r'_2$, where this order on $\N\times \N$ is later denoted by
    $(r_1,r_2)\succeq (r'_1,r'_2)$.
    Let us assume $\overline{\mathbb{O}(r_1, r_2)} \supset \mathbb{O}(r'_1, r'_2)$. Since the rank of matrices 
    does not increase 
    at the closure, we have $r_1 \geq r'_1$ and $r_2 \geq r'_2$. 
    Let us next assume $r_1 \geq r'_1$ and $r_2 \geq r'_2$. Since any element of $\mathbb{O}(r_1, r_2)$ 
    can be transformed to a Smith normal form by conjugations, we can construct a 1-parameter degeneration to any element of $\mathbb{O}(r'_1, r'_2)$ from $\mathbb{O}(r_1, r_2)$. Thus we have $\overline{\mathbb{O}(r_1, r_2)} \supset \mathbb{O}(r'_1, r'_2)$. 
    
    Therefore the irreducible components of $W(a_1, a_2, a_3)$ are parameterized by $(r_1, r_2) \in \mathbb N \times \mathbb N$ which satisfies $ 0 \leq r_1 \leq a_1$, $r_1+r_2 \leq a_2, 0 \leq r_2 \leq a_3$ and is maximal with respect to $\succeq$. 
    If $a_1 +a_3 < a_2$, $(a_1, a_3)$ is the maximum element. 
    If $a_1 +a_3 \geq a_2$, maximal elements are those $(r_1, r_2)$ which satisfy $ 0 \leq r_1 \leq a_1$, $0 \leq r_2 \leq a_3$ and $r_1+r_2=a_2$. 
\end{proof}

\begin{defn}\label{defn:2x2irred}
    For integers $r_1, r_2 \in \mathbb{N}$,  an algebraic subset $(\nu; r_1, r_2)$ of $E_C(\nu)$ is defined by properties 
    $\rank \begin{bmatrix} f_{12} & f_{13} \end{bmatrix}^T \leq r_1$ and 
    $\rank \begin{bmatrix} f_{24} & -f_{34} \end{bmatrix} \leq r_2$. 
\end{defn}

Using this definition the parametrization of irreducible components is rephrased as follows. 

\begin{cor}\label{cor:ex_irr}\hfill
    \begin{enumerate}
        \item For $\nu_1+\nu_4 \geq \nu_2+\nu_3$, 
        the elements of $B_C(\nu)$ are precisely given by $(\nu; r_1, r_2)$ for $(r_1, r_2) \in \mathbb N \times \mathbb N$ satisfying $
        r_1 + r_2 = \nu_2 + \nu_3
        ,~0 \leq r_1 \leq \nu_1
        ,~0 \leq r_2 \leq \nu_4$. 
        
        \item For $\nu_1+\nu_4 < \nu_2+\nu_3$, $E_C(\nu)=(\nu; \nu_1, \nu_4)$ is irreducible. 
    \end{enumerate}
\end{cor}

The explicit description of general representations in $(\nu; r_1, r_2)$ is derived in Appendix. 

\begin{rem}\label{ddv}
\begin{enumerate}
    \item The algebraic varieties $\overline{\mathbb{O}(r_1, r_2)}$ and $W(a_1, a_2, a_3)$ are called the double determinantal variety \cite{hesselink} and the variety of complexes \cite{DCS}, respectively. The parametrization of irreducible components of $W(a_1, a_2, a_3)$ is already given by Hesselink \cite[\S 6.5]{hesselink} and Musili and Seshadri \cite[\S 11]{MS}. A generalization of this theorem for larger $\overrightarrow{G}_{m}$ has not yet been understood well. 
    \item From \cite[Table 4]{hesselink}, we see that $E_C(\nu)$ is not equi-dimensional. 
    \item
    Geiß, Labardini-Fragoso and Schröer \cite[Proposition 7.1]{GLS} prove that the irreducible components of representation varieties of gentle algebras are described by the rank function. 
    Here the rank function is the function which evaluate the rank of matrices associated with arrows. 
    In our setting, we need a generalized rank function to describe the irreducible components, as the following example indicates. 
    Here by a generalized rank function we mean a function which evaluates the ranks of matrices associated with arrows or of some concatenation of these matrices \cite{km}. 
    Let $\nu=(1,1,1,2)$. Then the irreducible components are $(1,1,1,2;1,1)$ and $(1,1,1,2;0,2)$ (see Example \ref{ex:1112}). 
    The irreducible component $(1,1,1,2;1,1)$ cannot be characterized by the rank function since possible conditions given by the rank function are $\rank f_{ij}=0$ and $\rank f_{ij}\leq 1$. 
\end{enumerate}
\end{rem}

Before the general description, we present an example of the computation of Kashiwara operators. 
\begin{ex}\label{ex:1112}
    Here we compute Kashiwara operators for the dimension vector $\nu=(1,1,1,2)$ using the notation introduced at the beginning of this section and in Definition \ref{defn:2x2irred}. 
    By Corollary \ref{cor:ex_irr} (1), $B_C(\nu)$ consists of the following two irreducible components: 
    \begin{align*}
        \Lambda_1 &:=(1,1,1,2;1,1)=\{f\in E_C(\nu) ~\vert~ \rank \begin{bmatrix} f_{24} & -f_{34} \end{bmatrix} \leq 1\}, 
        \\
        \Lambda_2& := (1,1,1,2;0,2)=\{f\in E_C(\nu) ~\vert~ f_{12}=0, f_{13}=0\}. 
    \end{align*}

    Let us compute $\tilde e_1\Lambda_1$ and $\tilde e_1\Lambda_2$. First we need to compute $B_C(0,1,1,2)=B_C(\nu+\alpha_1)$. We have $B_C(0,1,1,2)=\{E_C(0,1,1,2)=(0,1,1,2;0,2)\}$ since in this case the commutativity relation is trivial. 
    Unwinding Definition \ref{defn:ef}, $\tilde e_1\Lambda$ is computed as follows. 
    Consider the set of all representations $\bar f$ that are obtained from $f \in \Lambda$ by replacing $\Bbbk^{\nu_1}$ by a codimension one subspace. 
    Then $\tilde e_1\Lambda$ is defined as the closure of this set if this closure is an irreducible component, and zero otherwise. 
    First consider $\tilde e_1\Lambda_1$. In this case representations $\bar f$ as above obtained from $f\in \Lambda_1$ are contained in the proper closed subset of $(0,1,1,2;0,2)$ defined by $\rank \begin{bmatrix} \bar f_{24} & -\bar f_{34} \end{bmatrix} \leq 1$. Hence we have $\tilde e_1\Lambda_1=0$. Now let us consider $\tilde e_1\Lambda_2$. In this case all $\bar f \in (0,1,1,2;0,2)$ is obtained from some $f \in \Lambda_2$. Hence we have $\tilde e_1\Lambda_2=(0,1,1,2;0,2)$.
    
    Let us compute $\tilde e_4\Lambda_2$ and $\tilde e_4\Lambda_2$. First we need to compute $B_C(1,1,1,1)=B_C(\nu+\alpha_4)$. We have $B_C(1,1,1,1)=\{E_C(1,1,1,1)=(1,1,1,1;1,1)\}$ by Corollary \ref{cor:ex_irr} (1). 
    Unwinding Definition \ref{defn:ef}, $\tilde e_4(\nu;r_1, r_2)$ is computed as follows. Consider the set of all representations $\bar f$ that are obtained from $f \in \Lambda$ satisfying $\rank \begin{bmatrix} f_{24} & -f_{34} \end{bmatrix}=r_2$ by replacing $\Bbbk^{\nu_4}$ by a codimension one subspace containing the image of $\begin{bmatrix} f_{24} & -f_{34} \end{bmatrix}$. 
    Then $\tilde e_4\Lambda$ is defined as the closure of this set if this closure is an irreducible component, and zero otherwise. 
    First consider $\tilde e_4\Lambda_1$. In this case representations $\bar f$ with $\rank \begin{bmatrix} \bar f_{24} & -\bar f_{34} \end{bmatrix}=1$ are obtained from $f\in \Lambda_1$ in the above way. These form a dense open subset of $(1,1,1,1;1,1)$. Hence we have $\tilde e_4\Lambda_1=(1,1,1,1;1,1)$. 
    Now let us consider $\tilde e_4\Lambda_2$. In this case there is no codimension one subspace of $\Bbbk^{\nu_4}=\Bbbk^2$ containing the image of $\begin{bmatrix} f_{24} & -f_{34} \end{bmatrix}$ with $\rank \begin{bmatrix} f_{24} & -f_{34} \end{bmatrix}=2$. Hence we have $\tilde e_4\Lambda_2=0$.
\end{ex}

\begin{thm}\label{thm:2x2ef} 
The action of the Kashiwara operators on $B_C$ is given as follows. 
    \begin{enumerate}
        \item 
        \begin{align*}
        &\tilde e_1 (\nu;r_1, r_2) =
        \begin{cases}
            (\nu+\alpha_1;r_1-1, r_2) & \text{if $\nu_1 + \nu_4 \leq \nu_2 + \nu_3$}, \\
            (\nu+\alpha_1;r_1, r_2) & \text{if $\nu_1 + \nu_4 > \nu_2 + \nu_3 $ and $\nu_1 > r_1$}, \\
            0 & \text{if $\nu_1 + \nu_4 > \nu_2 + \nu_3 $ and $\nu_1 = r_1$}.
        \end{cases}\\
        &\tilde e_2 (\nu;r_1, r_2) =
        \begin{cases}
            (\nu+\alpha_2;r_1, r_2) & \text{if $\nu_2 > r_1$ and $\nu_1 + \nu_4 < \nu_2 + \nu_3$}, \\
            (\nu+\alpha_2;r_1, r_2-1) & \text{if $\nu_2 > r_1$  and $\nu_1 + \nu_4 \geq \nu_2 + \nu_3$}, \\
            0 & \text{if $\nu_2 \leq r_1$}.
        \end{cases}\\
        &\tilde e_3 (\nu;r_1, r_2) =
        \begin{cases}
            (\nu+\alpha_3;r_1, r_2) & \text{if $\nu_3 > r_1$ and $\nu_1 + \nu_4 < \nu_2 + \nu_3$}, \\
            (\nu+\alpha_3;r_1, r_2-1) & \text{if $\nu_3 > r_1$ and $\nu_1 + \nu_4 \geq \nu_2 + \nu_3$}, \\
            0 & \text{if $\nu_3 \leq r_1$}.
        \end{cases}\\
        &\tilde e_4 (\nu;r_1, r_2) =
        \begin{cases}
            (\nu+\alpha_4;r_1, r_2)\quad\quad\!\! &  \text{if $\nu_4 > r_2$}, \\
            0 & \text{if $\nu_4 = r_2$}.
        \end{cases}
        \end{align*}

        \item 
        \begin{align*}
        &\tilde f_1 (\nu;r_1, r_2) =
        \begin{cases}
            (\nu-\alpha_1;r_1 + 1, r_2) & \text{if $\nu_1+\nu_4 < \nu_2+\nu_3$}, \\
            (\nu-\alpha_1;r_1 , r_2) & \text{if $\nu_1+\nu_4 \geq \nu_2+\nu_3$}.
        \end{cases}\\
        &\tilde f_2 (\nu;r_1, r_2) =
        \begin{cases}
            (\nu-\alpha_2;r_1, r_2) & \text{if $\nu_2 \geq r_1$ and $\nu_1 + \nu_4 \leq \nu_2 + \nu_3$}, \\
            (\nu-\alpha_2;r_1, r_2 + 1) & \text{if $\nu_2 \geq r_1$ and $\nu_1 + \nu_4 >  \nu_2 + \nu_3$}, \\
            0 & \text{if $\nu_2 < r_1$}.
        \end{cases}\\
        &\tilde f_3 (\nu;r_1, r_2) =
        \begin{cases}
            (\nu-\alpha_3;r_1, r_2) & \text{if $\nu_3 \geq r_1$ and $\nu_1 + \nu_4 \leq  \nu_2 + \nu_3$}, \\
            (\nu-\alpha_3;r_1, r_2 + 1) & \text{if $\nu_3 \geq r_1$ and $\nu_1 + \nu_4 >  \nu_2 + \nu_3$}, \\
            0 & \text{if $\nu_3 < r_1$}.
        \end{cases}\\
        &\tilde f_4 (\nu;r_1, r_2) =
        \begin{cases}
            (\nu-\alpha_4;r_1, r_2) \quad\quad\!\!& \text{if $\nu_1+\nu_4 \geq \nu_2+\nu_3$}, \\
            0 & \text{if $\nu_1+\nu_4 < \nu_2+\nu_3$}.
        \end{cases}
        \end{align*}
    \end{enumerate}
\end{thm}

\begin{proof}
    Let $\Lambda=(\nu;r_1, r_2)$. 
    Recall that the set $\varpi_1 \varpi^{-1}_2(\Lambda^{ 2})$ consists of $(\bar f_{\tau})_{\tau \in E}$ such that there is a short exact sequence $0 \to (\bar f_{\tau}) \to (f_{\tau}) \to S_i \to 0$ with $(f_{\tau}) \in \Lambda^{ 2}$, where $S_i$ is the simple representation associated to the vertex $i$. 

    \noindent $\bullet$ $\tilde e_1$:  
    Since there are no arrows ending at the vertex 1, we have $\Lambda = \Lambda^{ 2}$.     
    In the case $\nu_1 + \nu_4 \leq \nu_2 + \nu_3$, we have $r_1=\nu_1$, $r_2=\nu_4$ and 
    $\Lambda = E_C(\nu)$. 
    For any $(\bar f_{\tau}) \in (\nu+\alpha_1;\nu_1-1, \nu_4)$, there exists a split exact sequence $0 \to (\bar f_{\tau}) \to (\bar f_{\tau}) \oplus S_1 \to S_1 \to 0$ with $(\bar{f}_{\tau})\oplus S_1\in\Lambda$, which implies  $(\nu+\alpha_1;\nu_1-1, \nu_4) \subset \varpi_1 \varpi^{-1}_2(\Lambda)$. 
    Given that the left-hand side represents an irreducible component whereas the right-hand side is an irreducible set, $\tilde e_1(\nu;\nu_1,\nu_4)=(\nu+\alpha_1;\nu_1-1,\nu_4)$ holds. 
    For $\nu_1 + \nu_4 > \nu_2 + \nu_3$ and $\nu_1>r_1$,
    we can apply the same proof  by taking $(\bar f_{\tau}) \in (\nu+\alpha_1;r_1, r_2)$. 

    Assume  $\nu_1 + \nu_4 > \nu_2 + \nu_3$ and $\nu_1=r_1$. We now show that $\overline{\varpi_1 \varpi^{-1}_2(\Lambda)}$ is not an irreducible component of $E_{C}(\nu+\alpha_1)$. 
    Any representation $(\bar f_\tau)\in \varpi_1 \varpi^{-1}_2(\Lambda)$ satisfies $\rank 
    \begin{bmatrix}
    \bar f_{12} & \bar f_{13}\end{bmatrix}^{T}
    \leq r_1-1$ and $\rank 
    \begin{bmatrix}
    \bar f_{24} & -\bar f_{34}\end{bmatrix}
    \leq r_2$, implying  
    $\rank \begin{bmatrix}\bar f_{12} & \bar f_{13}\end{bmatrix}^{T} + 
    \rank \begin{bmatrix}\bar f_{24} & -\bar f_{34}\end{bmatrix} < \nu_2+\nu_3$.  
    Then, Proposition \ref{prop:irr2x2} shows  that $\overline{\varpi_1 \varpi^{-1}_2(\Lambda)}$ is not an irreducible component of $E_{C}(\nu+\alpha_1)$.

    \noindent $\bullet$ $\tilde e_2$ and $\tilde e_3$: 
    Assume $\nu_2 \leq r_1$. 
    In this case, we have
    $\Lambda^{ 2}=
    \{(f_\tau)~|~\varepsilon_2(f)=0\}$ 
    from $\nu_2\leq r_1\leq \nu_1$.
    Proposition \ref{prop:varpi} (3) implies $\varpi^{-1}_2(\Lambda^{ 2})=\emptyset$ and thus we have $\tilde e_2 \Lambda = 0$. 
    
    Assume $\nu_2>r_1$ and $\nu_1 + \nu_4 < \nu_2 + \nu_3$. In this case, we have $r_1=\nu_1$ and $r_2=\nu_4$. 
    Then, the set $\{(\bar{f}_\tau)\in(\nu+\alpha_2;\nu_1, \nu_4)~|~\rank \bar{f}_{12}=\nu_1\}$ becomes (nonempty) open. Then, since any element $(\bar f_{\tau})$ of this open set forms a short exact sequence $0 \to (\bar f_{\tau}) \to (\bar f_{\tau}) \oplus S_2 \to S_2 \to 0$ and $(\bar f_{\tau}) \oplus S_2\in\Lambda^{ 2}$, we have $(\bar f_{\tau})\in\varpi_1 \varpi^{-1}_2(\Lambda^{ 2})$. This implies $(\nu+\alpha_2; \nu_1, \nu_4) = \tilde e_2 \Lambda$.
    For $\nu_2>r_1$ and $\nu_1 + \nu_4 \geq \nu_2 + \nu_3 $, we can apply the same proof above by taking  
    $(\bar f_{\tau}) \in (\nu+\alpha_2;r_1, r_2-1)$ 
    with $\rank \bar f_{12}= r_1$. 
   The statement for $\tilde e_3$ is  proved in the same way as $\tilde e_2$.

    \noindent $\bullet$ $\tilde e_4$:  
    The set $\Lambda^{ 2}$ consists of $(f_{\tau})$ satisfying $\rank \begin{bmatrix}
        f_{24} & f_{34}
    \end{bmatrix} = r_2$.
    Since  $S_4$ is projective, there exists a short exact sequence $0 \to (\bar f_{\tau}) \to (f_{\tau}) \to S_4 \to 0$ if and only if $S_4$ is a direct summand of $(f_{\tau})$, which implies $\nu_4 >r_2$.

Then, for $\nu_4>r_2$, let us take $(\bar{f}_\tau)\in (\nu+\alpha_4;r_1,r_2)$ with $\rank \begin{bmatrix} \bar{f}_{24} & \bar{f}_{34}\end{bmatrix}=r_2$. 
Then, we can construct
a short exact sequence $0 \to (\bar f_{\tau}) \to (\bar f_{\tau}) \oplus S_4 \to S_4 \to 0$ with $(\bar f_{\tau}) \oplus S_4\in\Lambda^{ 2}$.
This concludes $(\nu+\alpha_4;r_1,r_2)=\tilde{e}_4\Lambda$.

    Now we consider $\tilde f_i$. In Lemma \ref{prop:ef}, we proved that for $\Lambda, \tilde \Lambda \in B_C$, $\tilde e_i \Lambda = \bar \Lambda$ if and only if $\Lambda = \tilde f_i \bar \Lambda $. 

    \noindent $\bullet$ $\tilde f_1$:  
    If $\nu_1+\nu_4 < \nu_2+\nu_3$, then we already proved $\tilde e_1 (\nu-\alpha_1;r_1+1, r_2) =  (\nu;r_1, r_2)$. 
    Therefore we have $\tilde f_1 (\nu;r_1, r_2) = (\nu-\alpha_1;r_1+1, r_2)$. 
    Similarly, for $\nu_1+\nu_4 \geq \nu_2+\nu_3$, we have $\tilde f_1 (\nu;r_1, r_2) = (\nu-\alpha_1;r_1, r_2)$
    from $\tilde e_1 (\nu-\alpha_1;r_1, r_2) =  (\nu;r_1, r_2)$.

    \noindent $\bullet$ $\tilde f_2$ and $\tilde f_3$:  
    If $\nu_2 \geq r_1$ and $\nu_1 + \nu_4 \leq \nu_2 + \nu_3$, then $\tilde e_2 (\nu-\alpha_2;r_1, r_2 )= (\nu;r_1, r_2)$. 
    If $\nu_2 \geq r_1$ and $\nu_1 + \nu_4 > \nu_2 + \nu_3$, then $\tilde e_2 (\nu-\alpha_2;r_1, r_2 + 1)= (\nu;r_1, r_2)$. 
    If $\nu_2 < r_1$, there are no irreducible components $\Lambda$ satisfying $\tilde e_2 \Lambda = (\nu;r_1, r_2)$.
    The statement for $\tilde f_3$ is proved in the same way as $\tilde f_2$. 

    \noindent $\bullet$ $\tilde f_4$:  
    If $\nu_1+\nu_4 \geq \nu_2+\nu_3$, then $\tilde e_4 (\nu-\alpha_4;r_1, r_2) = (\nu;r_1, r_2)$. 
    Therefore we have $\tilde f_4 (\nu;r_1, r_2) = (\nu-\alpha_4;r_1, r_2)$. 
    If $\nu_1+\nu_4 < \nu_2+\nu_3$, then we have $\nu_1=r_1$ and $\nu_4=r_2$ and  we also see that $E_C(\nu-\alpha_4)$ have only one irreducible component $(\nu-\alpha_4;\nu_1, \nu_4+1)$. Since we showed $\tilde e_4(\nu-\alpha_4;\nu_1, \nu_4+1)=0$, 
    we have $\tilde f_4(\nu;r_1, r_2)=0$.
\end{proof}

 The following $\ast$-version is proved in a similar way to Theorem \ref{thm:2x2ef}.

\begin{thm}\label{thm:2x2efstar}
The action of the $\ast$-Kashiwara operators is described as follows.
    \begin{enumerate}
        \item 
        \begin{align*}
        &\tilde e^*_1 (\nu;r_1, r_2) =
        \begin{cases}
            (\nu+\alpha_1;r_1, r_2) \quad\quad\! & \text{if $\nu_1 > r_1$}, \\
            0 & \text{if $\nu_1 = r_1$}.
        \end{cases}\\
        &\tilde e^*_2 (\nu;r_1, r_2) =
        \begin{cases}
            (\nu+\alpha_2;r_1, r_2) & \text{if $\nu_2 > r_2$ and $\nu_1+\nu_4<\nu_2+\nu_3$}, \\
            (\nu+\alpha_2;r_1-1, r_2) & \text{if $\nu_2 > r_2$ and $\nu_1+\nu_4 \geq \nu_2+\nu_3$}, \\
            0 & \text{if $\nu_2 \leq r_2$}.
        \end{cases}\\
        &\tilde e^*_3 (\nu;r_1, r_2) =
        \begin{cases}
            (\nu+\alpha_3;r_1, r_2) & \text{if $\nu_3 > r_2$ and $\nu_1+\nu_4<\nu_2+\nu_3$}, \\
            (\nu+\alpha_3;r_1-1, r_2) & \text{if $\nu_3 > r_2$ and $\nu_1+\nu_4 \geq \nu_2+\nu_3$}, \\
            0 & \text{if $\nu_3 \leq r_2$}.
        \end{cases}\\
        &\tilde e^*_4 (\nu;r_1, r_2) =
        \begin{cases}
            (\nu+\alpha_4;r_1, r_2-1) & \text{if $\nu_1+\nu_4 \leq \nu_2+\nu_3$}, \\
            (\nu+\alpha_4;r_1, r_2) & \text{if $\nu_1+\nu_4 > \nu_2+\nu_3$ and $\nu_4 > r_2$}, \\
            0 & \text{if $\nu_1+\nu_4 > \nu_2+\nu_3$ and $\nu_4 = r_2$}.
        \end{cases}
        \end{align*}
        
        \item 
        \begin{align*}
        &\tilde f^*_1 (\nu;r_1, r_2) =
        \begin{cases}
            (\nu-\alpha_1;r_1, r_2) \quad\quad\!& \text{if $\nu_1+\nu_4 \geq \nu_2+\nu_3$}, \\
            0 & \text{if $\nu_1+\nu_4 < \nu_2+\nu_3$}.
        \end{cases}\\
        &\tilde f^*_2 (\nu;r_1, r_2) =
        \begin{cases}
            (\nu-\alpha_2;r_1, r_2) & \text{if $\nu_2 \geq r_2$ and $\nu_1+\nu_4 \leq \nu_2+\nu_3$} \\
            (\nu-\alpha_2;r_1+1, r_2) & \text{if $\nu_2 \geq r_2$ and $\nu_1+\nu_4 > \nu_2+\nu_3$}, \\
            0 & \text{if $\nu_2 < r_2$}.
        \end{cases}\\
        &\tilde f^*_3 (\nu;r_1, r_2) =
        \begin{cases}
            (\nu-\alpha_3;r_1, r_2) & \text{if $\nu_3 \geq r_2$ and $\nu_1+\nu_4 \leq \nu_2+\nu_3$}, \\
            (\nu-\alpha_3;r_1+1, r_2) & \text{if $\nu_3 \geq r_2$ and $\nu_1+\nu_4 > \nu_2+\nu_3$}, \\
            0 & \text{if $\nu_3 < r_2$}.
        \end{cases}\\
        &\tilde f^*_4 (\nu;r_1, r_2) =
        \begin{cases}
            (\nu-\alpha_4;r_1, r_2+1) & \text{if $\nu_1+\nu_4 < \nu_2+\nu_3$}, \\
            (\nu-\alpha_4;r_1, r_2) & \text{if $\nu_1+\nu_4 \geq \nu_2+\nu_3$}.
        \end{cases}
        \end{align*}
    \end{enumerate}
\end{thm}

Next we compute some invariants of irreducible components. 

\begin{defn}\label{def:epsilonprime}
    For each $i \in I$, we define maps $\varepsilon_i', \varphi_i', \varepsilon_i^*{}', \varphi_i^*{}': B_C \to \mathbb{N} \sqcup \{\infty\}$ by 
    \begin{align*}
        \varepsilon'_i(\Lambda) &= \max \{ k \in \mathbb{N} \ \vert \ \tilde e^{k}_i \Lambda \neq 0 \}, &
        \varphi'_i(\Lambda) &= \sup \{ k \in \mathbb{N} \  \vert  \ \tilde f^{k}_i \Lambda \neq 0 \}, \\
        \varepsilon_i^*{}'(\Lambda) & = \max \{ k \in \mathbb{N} \ \vert \ \tilde e^{*k}_i \Lambda \neq 0 \}, &
        \varphi_i^*{}'(\Lambda) &= \sup \{ k \in \mathbb{N} \ \vert \ \tilde f^{*k}_i \Lambda \neq 0 \}. 
    \end{align*}
\end{defn}

\begin{prop}\label{prop:varepsilon2x2} 
The values of the functions $\varepsilon_i, \varphi_i, \varepsilon_i', \varphi_i'$ are computed as follows. 
    \begin{enumerate}
        \item
        \begin{align*}
        \varepsilon_1 (\nu;r_1, r_2) &= \nu_1. &
        \varphi_1 (\nu;r_1, r_2) &= -\nu_1+\nu_2+\nu_3. \\
        \varepsilon_2 (\nu;r_1, r_2) &=
        \begin{cases}
            \nu_2-r_1& \text{if $\nu_2>r_1$}, \\
            0 & \text{if $\nu_2\leq r_1$}.
        \end{cases} &
        \varphi_2 (\nu;r_1, r_2) &=
        \begin{cases}
            -\nu_2+\nu_1+\nu_4-r_1 & \text{if $\nu_2>r_1$}, \\
            -2\nu_2+\nu_1+\nu_4 & \text{if $\nu_2\leq r_1$}.
        \end{cases}\\
        \varepsilon_3 (\nu;r_1, r_2) &=
        \begin{cases}
            \nu_3-r_1 & \text{if $\nu_3>r_1$}, \\
            0 & \text{if $\nu_3\leq r_1$}.
        \end{cases}&
        \varphi_3 (\nu;r_1, r_2) &=
        \begin{cases}
            -\nu_3+\nu_1+\nu_4-r_1 & \text{if $\nu_3>r_1$}, \\
            -2\nu_3+\nu_1+\nu_4 & \text{if $\nu_3\leq r_1$}.
        \end{cases}\\
        \varepsilon_4 (\nu;r_1, r_2) &= \nu_4-r_2. &
        \varphi_4 (\nu;r_1, r_2) &= -\nu_4+\nu_2+\nu_3-r_2.
        \end{align*}

        \item
        \begin{align*}
        \varepsilon'_1 (\nu;r_1, r_2) &= 
        \begin{cases}
            \nu_1 & \text{if $\nu_4=r_2$}, \\
            \nu_1-r_1 & \text{if $\nu_4>r_2$}.
        \end{cases}&
        \varphi'_1 (\nu;r_1, r_2) &= \infty. \\
        \varepsilon'_2 (\nu;r_1, r_2) &=
        \begin{cases}
            \nu_2-r_1& \text{if $\nu_2>r_1$}, \\
            0 & \text{if $\nu_2\leq r_1$}.
        \end{cases}&
        \varphi'_2 (\nu;r_1, r_2) &=
        \begin{cases}
            \infty& \text{if $\nu_2\geq r_1$}, \\
            0 & \text{if $\nu_2< r_1$}.
        \end{cases}\\
        \varepsilon'_3 (\nu;r_1, r_2) &=
        \begin{cases}
            \nu_3-r_1 & \text{if $\nu_3>r_1$}, \\
            0 & \text{if $\nu_3\leq r_1$}.
        \end{cases}&
        \varphi'_3 (\nu;r_1, r_2) &=
        \begin{cases}
            \infty & \text{if $\nu_3\geq r_1$}, \\
            0 & \text{if $\nu_3< r_1$}.
        \end{cases}\\
        \varepsilon'_4 (\nu;r_1, r_2) &=   \nu_4-r_2. &
        \varphi'_4 (\nu;r_1, r_2) &= 
        \begin{cases}
            \infty& \text{if $\nu_1+\nu_4 \geq \nu_2+\nu_3$}, \\
            0 & \text{if $\nu_1+\nu_4 < \nu_2+\nu_3$}.
        \end{cases}
        \end{align*}
    \end{enumerate}
\end{prop}

\begin{proof}
    (1) 
    Since there are no arrows ending at the vertex 1, we have $\varepsilon_1(\nu;r_1, r_2)= \nu_1$. 
    
    For $f \in (\nu;r_1, r_2)$, we have $\rank f_{12} \leq r_1$ by definition. Since the codomain of $f_{12}$ is $\Bbbk^{\nu_2}$, we have $\rank f_{12} \leq \nu_2$. 
    By taking $f_{12}$ with $\rank f_{12} = \min\{r_1, \nu_2\}$,  $f_{13}=0$, $f_{24}=0$ and $f_{34}=0$,  we have $f \in (\nu;r_1, r_2)$. This implies the equality $\varepsilon_2(\nu;r_1, r_2)= \nu_2-\min\{r_1, \nu_2\}$. The statement for $\varepsilon_3(\nu;r_1, r_2)$ is proved in the same way. 
    
    For $f \in (\nu;r_1, r_2)$, we have $\rank \begin{bmatrix} f_{24} & -f_{34} \end{bmatrix} \leq r_2$. Since  $r_2 \leq \nu_2+\nu_3$ and $r_2 \leq \nu_4$ hold, we can take $f_{24}$, $f_{34}$ with $\rank \begin{bmatrix} f_{24} & -f_{34} \end{bmatrix} = r_2$. By setting $f_{12}=0$ and $f_{13}=0$,  this $f$ belongs to $(\nu;r_1, r_2)$. Hence we have $\varepsilon_4(\nu;r_1, r_2)= \nu_4-r_2$.
    
    The statemants for $\varphi_i$ follow from what we have proved and equalities $\langle h_1, \nu \rangle =-2\nu_1 +\nu_2 + \nu_3$, $\langle h_2, \nu \rangle =-2\nu_2 +\nu_1 + \nu_4$, $\langle h_3, \nu \rangle =-2\nu_3 +\nu_1 + \nu_4$ and  $\langle h_4, \nu \rangle =-2\nu_4 +\nu_2 + \nu_3$. 

    (2) Assume $\nu_1+\nu_4 \leq \nu_2+\nu_3$. In this case, we have $\nu_1=r_1$ and $\nu_4=r_2$ , hence $\tilde e^{\nu_1}_1(\nu;r_1, r_2)= (0, \nu_2, \nu_3, \nu_4; 0, r_2)$ and $\tilde e^{\nu_1+1}_1(\nu;r_1, r_2)=0$ by Theorem \ref{thm:2x2ef}. This implies the equality $\varepsilon'_1(\nu;r_1, r_2)= \nu_1$. Assume $\nu_1+\nu_4 > \nu_2+\nu_3$. 
    In this case, we have $\tilde e^{\nu_1-r_1}_1(\nu;r_1, r_2)= (r_1, \nu_2, \nu_3, \nu_4; r_1, r_2)$. 
    If $\nu_4=r_2$, we have $r_1+\nu_4=\nu_2+\nu_3$ and we can further apply $\tilde e_1$ and obtain $\tilde e^{\nu_1}_1(\nu;r_1, r_2)= (0, \nu_2, \nu_3, \nu_4; 0, r_2)$ and $\tilde e^{\nu_1+1}_1(\nu;r_1, r_2)=0$. This implies the equality $\varepsilon'_1(\nu;r_1, r_2)= \nu_1$. If $\nu_4>r_2$, then we have $r_1+\nu_4>r_1+r_2 = \nu_2+\nu_3$ and therefore $\tilde e_1^{\nu_1-r_1+1}(r_1, \nu_2, \nu_3, \nu_4; r_1, r_2)=0$. This implies the equality $\varepsilon'_1(\nu;r_1, r_2)= \nu_1-r_1$. 

    By Theorem \ref{thm:2x2ef}, we have $\tilde e_2(\nu; r_1, r_2)\neq 0$ if and only if we have $\nu_2>r_1$. This immediately implies the statement for $\varepsilon'_2(\nu;r_1, r_2)$. The statement for $\varepsilon'_3(\nu;r_1, r_2)$ is proved in the same way. Theorem \ref{thm:2x2ef} directly implies the statement for $\varepsilon'_4(\nu;r_1, r_2)$. 

    By Theorem \ref{thm:2x2ef}, we always have $\tilde f_1(\nu;r_1, r_2) \neq 0$ and hence we have the equality $\varphi'_1(\nu;r_1, r_2)= \infty$. The condition $\tilde f_2(\nu; r_1, r_2)\neq 0$ is equivalent to the condition $\nu_2\geq r_1$. This implies the statement for $\varphi'_2(\nu;r_1, r_2)$. The statement for $\varphi'_3(\nu;r_1, r_2)$ is proved in the same way. The statement for $\varphi'_4(\nu;r_1, r_2)$ follows immediately from Theorem \ref{thm:2x2ef}. 
\end{proof}

The following proposition is proved in the same way as the previous one. 

\begin{prop}\label{prop:varepsilonstar2x2}
The values of the functions $\varepsilon_i^*{}, \varphi_i^*{}, \varepsilon_i^*{}', \varphi_i^*{}'$ are computed as follows. 
    \begin{enumerate}
        \item
        \begin{align*}
        \varepsilon^*_1 (\nu;r_1, r_2) &= \nu_1-r_1. &
        \varphi^*_1 (\nu;r_1, r_2) &=  -\nu_1+\nu_2+\nu_3-r_1. \\
        \varepsilon^*_2 (\nu;r_1, r_2) &=
        \begin{cases}
            \nu_2-r_2& \text{if $\nu_2>r_2$}, \\
            0 & \text{if $\nu_2\geq r_2$}.
        \end{cases}&
        \varphi^*_2 (\nu;r_1, r_2) &=
        \begin{cases}
            -\nu_2+\nu_1+\nu_4-r_2 & \text{if $\nu_2>r_2$}, \\
            -2\nu_2+\nu_1+\nu_4 & \text{if $\nu_2\leq r_2$}.
        \end{cases}\\
        \varepsilon^*_3 (\nu;r_1, r_2) &=
        \begin{cases}
            \nu_3-r_2 & \text{if $\nu_3>r_2$}, \\
            0 & \text{if $\nu_3\geq r_2$}.
        \end{cases}&
        \varphi^*_3 (\nu;r_1, r_2) &=
        \begin{cases}
            -\nu_3+\nu_1+\nu_4-r_2 & \text{if $\nu_3>r_2$}, \\
            -2\nu_3+\nu_1+\nu_4 & \text{if $\nu_3\leq r_2$}.
        \end{cases}\\
        \varepsilon^*_4 (\nu;r_1, r_2) &=\nu_4. &
        \varphi^*_4 (\nu;r_1, r_2) &= -\nu_4+\nu_2+\nu_3. 
        \end{align*}

        \item
        \begin{align*}
        \varepsilon_1^*{}' (\nu;r_1, r_2) &=\nu_1-r_1. &
        \varphi_1^*{}' (\nu;r_1, r_2) &= 
        \begin{cases}
            \infty& \text{if $\nu_1+\nu_4 \geq \nu_2+\nu_3$}, \\
            0 & \text{if $\nu_1+\nu_4 < \nu_2+\nu_3$}.
        \end{cases}\\
        \varepsilon_2^*{}' (\nu;r_1, r_2) &=
        \begin{cases}
            \nu_2-r_1& \text{if $\nu_2>r_2$}, \\
            0 & \text{if $\nu_2\leq r_2$}.
        \end{cases}&
        \varphi_2^*{}' (\nu;r_1, r_2) &=
        \begin{cases}
            \infty& \text{if $\nu_2\geq r_2$}, \\
            0 & \text{if $\nu_2< r_2$}.
        \end{cases}\\
        \varepsilon_3^*{}' (\nu;r_1, r_2) &=
        \begin{cases}
            \nu_3-r_1 & \text{if $\nu_3>r_2$}, \\
            0 & \text{if $\nu_3\leq r_2$}.
        \end{cases}&
        \varphi_3^*{}' (\nu;r_1, r_2) &=
        \begin{cases}
            \infty & \text{if $\nu_3\geq r_2$}, \\
            0 & \text{if $\nu_3< r_2$}.
        \end{cases}\\
        \varepsilon_4^*{}' (\nu;r_1, r_2) &= 
        \begin{cases}
            \nu_4-r_2& \text{if $\nu_1>r_1$}, \\
            \nu_4 & \text{if $\nu_1=r_1$}.
        \end{cases} &
        \varphi_4^*{}' (\nu;r_1, r_2) &= \infty. 
        \end{align*}
    \end{enumerate}
\end{prop}

Finally, we prove that $B_C$ cannot be embedded into $B({\infty})$, as well as showing some fundamental properties about $B_C$. 
We recall that the crystal is called \emph{upper seminormal} if 
$\varepsilon_i(\Lambda) = \varepsilon_i'(\Lambda)$
holds for any $i$ and $\Lambda$.

\begin{cor}\label{cor:BC2x2property}\hfill
\begin{enumerate}
    \item If $\overrightarrow{G}_{(m_1,\dots,m_d)}$ is not one-parameter (i.e., when at least two of the $m_i$ are greater than 1), then there exists no embedding of crystals $B_C \to B(\infty)$. 
    \item If $\overrightarrow{G}_{(m_1,\dots,m_d)}$ is not one-parameter (i.e., when at least two of the $m_i$ are greater than 1), then the crystal $B_C$ is not upper seminormal. 
    \item For $\overrightarrow{G}_{(2,2)}$, the crystal $B_C$ is connected. 
\end{enumerate}
\end{cor}
\begin{proof}
    For (1) and (2), we can reduce the problem to $\overrightarrow{G}_{(2,2)}$. 
    
    (1) 
    Let $u_C\in B_C(0)$ ($u_\infty\in B(\infty,0)$, respectively) be the unique irreducible component with the zero dimension vector. 
    If there is an embedding $B_C \to B(\infty)$ then $u_C$ is sent to $u_\infty$ since they are unique weight zero elements. 
    By Theorem \ref{thm:2x2ef}, we have 
    $\tilde f^{}_{3} \tilde f^{2}_{1} \tilde f^{}_{3} \tilde f^{2}_{4} u_C
    =
    \tilde f^{2}_{1} \tilde f^{2}_{3} \tilde f^{2}_{4} u_C
    $ in $B_C$. 
    On the other hand, using polyhedral realizations of Nakashima and Zelevinsky \cite[\S 6]{NZ}, 
    we have 
    $\tilde f^{}_{3} \tilde f^{2}_{1} \tilde f^{}_{3} \tilde f^{2}_{4} u_{\infty}
    \neq
    \tilde f^{2}_{1} \tilde f^{2}_{3} \tilde f^{2}_{4} u_{\infty}
    $ in $B(\infty)$.
    This implies that there does not exist a crystal embedding $B_C \to B({\infty})$. 

    (2) 
    Let $\Lambda$ be $(1,1,1,2;1,1)$.
    On one hand, we have $\varepsilon_1(\Lambda) = 1$ since $\nu_1=1$. 
    On the other hand, we have $\varepsilon'_1(\Lambda) = 0$ by Proposition \ref{prop:varepsilon2x2}. 
    
    (3) 
    By Theorem \ref{thm:2x2ef}, we have 
    $(\nu;r_1, r_2)=\tilde f_4^{\nu_4-r_2}\tilde f_1^{\nu_1}\tilde f_2^{\nu_2}\tilde f_3^{\nu_3}\tilde f_4^{r_2}u_C$.
    This implies that the crystal is connected. 
\end{proof}

\appendix
\section{Appendix}

The bound quiver $\overrightarrow{G}_{(2,2)}$ has eleven isomorphism classes of indecomposable representations. Let us denote 
by $M(\nu_1, \nu_2, \nu_3, \nu_4)$ the representation of $\overrightarrow{G}_{(2,2)}$ with the dimension vector 
$(\nu_1, \nu_2, \nu_3, \nu_4)$
consisting of 0 and 1, with morphisms to be identity if possible and zero otherwise. 
Then the eleven non-isomorphic indecomposable representations are listed as 
$M_1=M(1, 0, 0, 0)=\Mon$, 
$M_2=M(0, 1, 0, 0)=\Mtw$, 
$M_3=M(0, 0, 1, 0)=\Mth$, 
$M_4=M(0, 0, 0, 1)=\Mfo$, 
$M_5=M(1, 1, 0, 0)=\Mfi$, 
$M_6=M(1, 0, 1, 0)=\Msi$, 
$M_7=M(0, 1, 0, 1)=\Mse$, 
$M_8=M(0, 0, 1, 1)=\Mei$, 
$M_9=M(1, 1, 1, 0)=\Mni$, 
$M_{10}=M(0, 1, 1, 1)=\Mte$, 
$M_{11}=M(1, 1, 1, 1)=\Mel$, 
where  we introduced the graphical notation like $\Mte=M_{10}$ for visual clarity. Also, the Auslander-Reiten quiver for $\overrightarrow{G}_{(2,2)}$ is given by
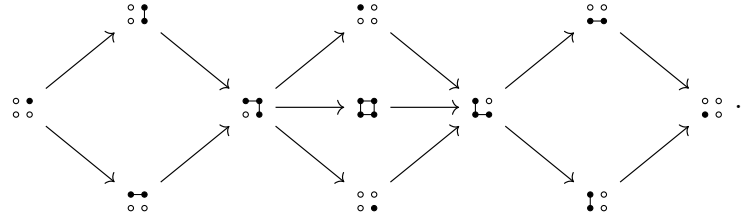
\[
    \begin{tikzcd}
     & \Mse\arrow[dr] &  & \Mth\arrow[dr] &  & \Mfi\arrow[dr] & \\
    \Mfo\arrow[ur]\arrow[dr] &  & \Mte\arrow[ur]\arrow[dr]\arrow[r]& \Mel\arrow[r] & \Mni\arrow[ur]\arrow[dr] &  & \Mon\\
     & \Mei\arrow[ur] &  & \Mtw\arrow[ur] &  & \Msi\arrow[ur] & 
    \end{tikzcd}.
\]

Since the bound quiver $\overrightarrow{G}_{(2,2)}$ is of finite representation type, for each irreducible component of the representation variety, there is an open subset consisting of one isomorphism class of representations.
In this section, we determine this general representation for all irreducible components. 

The main technical fact of this appendix is the following result by Crawley-Boevey and Schröer \cite{CBS}. Let A be a $\Bbbk$-algebra with orthogonal idempotents $e_1, \ldots, e_\ell$ satisfying $e_1+\cdots+e_\ell=1$. 
For a dimension vector ${\bf d}=(\nu_1, \dots, \nu_\ell) \in \mathbb{N}^{\ell}$, define $\GL_{\bf d}(\Bbbk) := \prod_i \GL_{\nu_i}(\Bbbk)$ and $\textrm{mod}^{\bf d}_{A}(\Bbbk)$ to be the set of $\Bbbk$-algebra homomorphisms $A \to \textrm{End}_{\Bbbk}(\Bbbk^{\nu_1+\cdots+\nu_\ell})$ sending each $e_i$ to the matrix whose diagonal $\nu_i \times \nu_i$ block is the identity with zeros for all other blocks. 
In our setting, when we set $A$ to be the path algebra defined by ${\overrightarrow{G}_{m}}$, we have $\ell=|I|$ and $\textrm{mod}^{\bf d}_{A}(\Bbbk)=E_C(\nu)$ for $\nu=-\sum_{i\in I} \nu_i\alpha_i\in -{\rm Q}_+$. 
For dimension vectors ${\bf d}^j\in \mathbb{N}^{\ell}$ ($j=1,\dots,t$) and algebraic sets $C_j \subset \textrm{mod}^{{\bf d}^j}_{A}(\Bbbk)$, let us define $C_1 \oplus \cdots \oplus C_t \subset \textrm{mod}^{{\bf d}^1 + \dots + {\bf d}^t}_{A}(\Bbbk)$ to be the image of the map $\GL_{{\bf d}^1 + \dots + {\bf d}^t}(\Bbbk) \times C_1 \times \cdots \times C_t \to \textrm{mod}^{{\bf d}^1 + \dots + {\bf d}^t}_{A}(\Bbbk)$ given by the direct sum and conjugation. 

\begin{thm}\cite[Theorem 1.1.]{CBS}\label{CBS1}
    If $C$ is an irreducible component in ${\rm mod}^{{\bf d}}_{A}(\Bbbk)$, then 
    $C=\overline{C_1 \oplus \cdots \oplus C_t}$ for some irreducible components $C_j$ of ${\rm mod}^{{\bf d}^j}_{A}(\Bbbk)$, with the property that the general representation in each $C_j$ is indecomposable. Moreover $C_1,\dots,C_t$ are uniquely determined by this up to reordering. 
\end{thm}

\begin{thm}\cite[Theorem 1.2.]{CBS}\label{CBS2}
    If  $C_j$ is an irreducible component of ${\rm mod}^{{\bf d}^j}_{A}(\Bbbk)$ for each $j$, then $\overline{C_1 \oplus \cdots \oplus C_t}$ is an irreducible component of ${\rm mod}^{{\bf d}^1+\dots+{\bf d}^t}_{A}(\Bbbk)$ if and only if 
    $\min\left\{ \dim{\rm Ext}^1_A(M_i, M_j) ~\vert~  M_i \in C_i, M_j \in C_j\right\}=0$ for all $i \neq j$. 
\end{thm}

Our main result in this appendix is the following. 

\begin{thm}\label{prop:2x2general}
    For $\nu = - \sum_{i \in I} \nu_i \alpha_i \in -Q_+$, 
    the general representation in $(\nu;r_1,r_2)\subset E_C(\nu)$ is isomorphic to the following. 
    Assume $\nu_1+\nu_4 \geq \nu_2+\nu_3$. 
    \begin{enumerate}
        \item If $\nu_{1} \leq \nu_{2} \leq \nu_{3}$, 
        \[
        (\Mon)^{\oplus \nu_1-r_1} \oplus (\Mfo)^{\oplus \nu_4-r_2} \oplus (\Mse)^{\oplus \nu_2-r_1} \oplus (\Mei)^{\oplus \nu_3-r_1} \oplus (\Mel)^{\oplus r_1}
        \]

        \item If $\nu_{2} \leq \nu_{1} \leq \nu_{3}$ and $ r_{1} \leq \nu_{2}$, 
        \[
        (\Mon)^{\oplus \nu_1-r_1} \oplus (\Mfo)^{\oplus \nu_4-r_2} \oplus (\Mse)^{\oplus \nu_2-r_1} \oplus (\Mei)^{\oplus \nu_3-r_1} \oplus (\Mel)^{\oplus r_1}
        \]

        \item If $\nu_{2} \leq \nu_{1} \leq \nu_{3}$ and $\nu_{2} \leq r_{1}$, 
        \[
        (\Mon)^{\oplus \nu_1-r_1} \oplus (\Mfo)^{\oplus \nu_4-r_2} \oplus (\Msi)^{\oplus r_1-\nu_2} \oplus (\Mei)^{\oplus \nu_3-r_1} \oplus (\Mel)^{\oplus \nu_2}
        \]

        \item If $\nu_{2} \leq \nu_{3} \leq \nu_{1}$ and $ r_{1} \leq \nu_{2}$, 
        \[
        (\Mon)^{\oplus \nu_1-r_1} \oplus (\Mfo)^{\oplus \nu_4-r_2} \oplus (\Mse)^{\oplus \nu_2-r_1} \oplus (\Mei)^{\oplus \nu_3-r_1} \oplus (\Mel)^{\oplus r_1}
        \]

        \item If $\nu_{2} \leq \nu_{3} \leq \nu_{1}$ and $\nu_{2} \leq r_{1} \leq \nu_{3}$, 
        \[
        (\Mon)^{\oplus \nu_1-r_1} \oplus (\Mfo)^{\oplus \nu_4-r_2} \oplus (\Msi)^{\oplus r_1-\nu_2} \oplus (\Mei)^{\oplus \nu_3-r_1} \oplus (\Mel)^{\oplus \nu_2}
        \]

        \item If $\nu_{2} \leq \nu_{3} \leq \nu_{1}$ and $\nu_{3} \leq r_{1}$, 
        \[
        (\Mon)^{\oplus \nu_1-r_1} \oplus (\Mfo)^{\oplus \nu_4-r_2} \oplus (\Msi)^{\oplus r_1-\nu_2} \oplus (\Mfi)^{\oplus r_1-\nu_3} \oplus (\Mel)^{\oplus r_2}
        \]

        \item If $\nu_{1} \leq \nu_{3} \leq \nu_{2}$, 
        \[
        (\Mon)^{\oplus \nu_1-r_1} \oplus (\Mfo)^{\oplus \nu_4-r_2} \oplus (\Mse)^{\oplus \nu_2-r_1} \oplus (\Mei)^{\oplus \nu_3-r_1} \oplus (\Mel)^{\oplus r_1}
        \]

    \item If $\nu_{3} \leq \nu_{1} \leq \nu_{2}$ and $ r_{1} \leq \nu_{3}$, 
        \[
        (\Mon)^{\oplus \nu_1-r_1} \oplus (\Mfo)^{\oplus \nu_4-r_2} \oplus (\Mse)^{\oplus \nu_2-r_1} \oplus (\Mei)^{\oplus \nu_3-r_1} \oplus (\Mel)^{\oplus r_1}
        \]
        
        \item If $\nu_{3} \leq \nu_{1} \leq \nu_{2}$ and $\nu_{3} \leq r_{1}$, 
        \[
        (\Mon)^{\oplus \nu_1-r_1} \oplus (\Mfo)^{\oplus \nu_4-r_2} \oplus (\Mfi)^{\oplus r_1-\nu_3} \oplus (\Mse)^{\oplus \nu_2-r_1} \oplus (\Mel)^{\oplus \nu_3}
        \]

        \item If $\nu_{3} \leq \nu_{2} \leq \nu_{1}$ and $r_{1} \leq \nu_{3}$, 
        \[
        (\Mon)^{\oplus \nu_1-r_1} \oplus (\Mfo)^{\oplus \nu_4-r_2} \oplus (\Mse)^{\oplus \nu_2-r_1} \oplus (\Mei)^{\oplus \nu_3-r_1} \oplus (\Mel)^{\oplus r_1}
        \]

        \item If $\nu_{3} \leq \nu_{2} \leq \nu_{1}$ and $\nu_{3} \leq r_{1} \leq \nu_{2} $, 
        \[
        (\Mon)^{\oplus \nu_1-r_1} \oplus (\Mfo)^{\oplus \nu_4-r_2} \oplus (\Mfi)^{\oplus r_1-\nu_3} \oplus (\Mse)^{\oplus \nu_2-r_1} \oplus (\Mel)^{\oplus \nu_3}
        \]

        \item If $\nu_{3} \leq \nu_{2} \leq \nu_{1}$ and $\nu_{2} \leq r_{1}$, 
        \[
        (\Mon)^{\oplus \nu_1-r_1} \oplus (\Mfo)^{\oplus \nu_4-r_2} \oplus (\Mfi)^{\oplus r_1-\nu_3} \oplus (\Msi)^{\oplus r_1-\nu_2} \oplus (\Mel)^{\oplus r_2}
        \]

    \end{enumerate}

    Assume $\nu_1+\nu_4 < \nu_2+\nu_3$ and $\nu_2 \leq \nu_3$ (in this case $r_1=\nu_1$ and $r_2=\nu_4$). 
    \begin{enumerate}
        \item If $\nu_{4} \leq \nu_{1} \leq \nu_{2} \leq \nu_{3}$, 
        $
        (\Mtw)^{\oplus \nu_2-\nu_1} \oplus (\Mth)^{\oplus \nu_3-\nu_1} \oplus (\Mni)^{\oplus \nu_1-\nu_4} \oplus (\Mel)^{\oplus \nu_4}
        $

        \item If $\nu_{4} \leq \nu_{2} \leq \nu_{1} \leq \nu_{3}$, 
        $
        (\Mth)^{\oplus \nu_3-\nu_1} \oplus (\Msi)^{\oplus \nu_1-\nu_2} \oplus (\Mni)^{\oplus \nu_2-\nu_4} \oplus (\Mel)^{\oplus \nu_4}
        $

        \item If $\nu_{4} \leq \nu_{2} \leq \nu_{3} \leq \nu_{1}$, 
        $
        (\Mfi)^{\oplus \nu_1-\nu_3} \oplus (\Msi)^{\oplus \nu_1-\nu_2} \oplus (\Mni)^{\oplus \nu_2+\nu_3-\nu_1-\nu_4} \oplus (\Mel)^{\oplus \nu_4}
        $
        
        \item If $\nu_{2} \leq \nu_{4} \leq \nu_{1} \leq \nu_{3}$, 
        $(\Mth)^{\oplus \nu_2+\nu_3-\nu_1-\nu_4} \oplus (\Msi)^{\oplus \nu_1-\nu_2} \oplus (\Mei)^{\oplus \nu_4-\nu_2} \oplus (\Mel)^{\oplus \nu_2}$
        
        \item If $\nu_{1} \leq \nu_{4} \leq \nu_{2} \leq \nu_{3}$, 
        $(\Mtw)^{\oplus \nu_2-\nu_4} \oplus (\Mth)^{\oplus \nu_3-\nu_4} \oplus (\Mte)^{\oplus \nu_4-\nu_1} \oplus (\Mel)^{\oplus \nu_1}$

        \item If $\nu_{1} \leq \nu_{2} \leq \nu_{4} \leq \nu_{3}$, 
        $(\Mth)^{\oplus \nu_3-\nu_4} \oplus (\Mei)^{\oplus \nu_4-\nu_2} \oplus (\Mte)^{\oplus \nu_2-\nu_1} \oplus (\Mel)^{\oplus \nu_1}$

        \item If $\nu_{1} \leq \nu_{2} \leq \nu_{3} \leq \nu_{4}$
        $(\Mse)^{\oplus \nu_4-\nu_3} \oplus (\Mei)^{\oplus \nu_4-\nu_2} \oplus (\Mte)^{\oplus \nu_2+\nu_3-\nu_1-\nu_4} \oplus (\Mel)^{\oplus \nu_1}$

        \item If $\nu_2 \leq \nu_1 \leq \nu_4 \leq \nu_3$, 
        $(\Mth)^{\oplus \nu_2+\nu_3-\nu_1-\nu_4} \oplus (\Msi)^{\oplus \nu_1-\nu_2} \oplus (\Mei)^{\oplus \nu_4-\nu_2} \oplus (\Mel)^{\oplus \nu_2}$
    \end{enumerate}
    Note that because of the assumption $\nu_1+\nu_4 < \nu_2+\nu_3$, the cases $\nu_{2} \leq \nu_{4} \leq \nu_{3} \leq \nu_{1}$, $\nu_{2} \leq \nu_{3} \leq \nu_{4} \leq \nu_{1}$, 
    $\nu_{2} \leq \nu_{1} \leq \nu_{3} \leq \nu_{4}$ and $\nu_{2} \leq \nu_{3} \leq \nu_{1} \leq \nu_{4}$ do not occur.

    Assume $\nu_1+\nu_4 < \nu_2+\nu_3$ and $\nu_3 \leq \nu_2$ (in this case $r_1=\nu_1$ and $r_2=\nu_4$). 
    \begin{enumerate}
        \item If $\nu_{4} \leq \nu_{1} \leq \nu_{3} \leq \nu_{2}$, 
        $
        (\Mtw)^{\oplus \nu_2-\nu_1} \oplus (\Mth)^{\oplus \nu_3-\nu_1} \oplus (\Mni)^{\oplus \nu_1-\nu_4} \oplus (\Mel)^{\oplus \nu_4}
        $

        \item If $\nu_{4} \leq \nu_{3} \leq \nu_{1} \leq \nu_{2}$, 
        $
        (\Mtw)^{\oplus \nu_2-\nu_1} \oplus (\Mfi)^{\oplus \nu_1-\nu_3} \oplus (\Mni)^{\oplus \nu_3-\nu_4} \oplus (\Mel)^{\oplus \nu_4}
        $

        \item If $\nu_{4} \leq \nu_{3} \leq \nu_{2} \leq \nu_{1}$, 
        $
        (\Mfi)^{\oplus \nu_1-\nu_3} \oplus (\Msi)^{\oplus \nu_1-\nu_2} \oplus (\Mni)^{\oplus \nu_2+\nu_3-\nu_1-\nu_4} \oplus (\Mel)^{\oplus \nu_4}
        $
        
        \item If $\nu_{3} \leq \nu_{4} \leq \nu_{1} \leq \nu_{2}$, 
        $(\Mtw)^{\oplus \nu_2+\nu_3-\nu_1-\nu_4} \oplus (\Mfi)^{\oplus \nu_1-\nu_3} \oplus (\Mse)^{\oplus \nu_4-\nu_3} \oplus (\Mel)^{\oplus \nu_3}$
        
        \item If $\nu_{1} \leq \nu_{4} \leq \nu_{3} \leq \nu_{2}$, 
        $(\Mtw)^{\oplus \nu_2-\nu_4} \oplus (\Mth)^{\oplus \nu_3-\nu_4} \oplus (\Mte)^{\oplus \nu_4-\nu_1} \oplus (\Mel)^{\oplus \nu_1}$

        \item If $\nu_{1} \leq \nu_{3} \leq \nu_{4} \leq \nu_{2}$, 
        $(\Mtw)^{\oplus \nu_2-\nu_4} \oplus (\Mse)^{\oplus \nu_4-\nu_3} \oplus (\Mte)^{\oplus \nu_3-\nu_1} \oplus (\Mel)^{\oplus \nu_1}$

        \item If $\nu_{1} \leq \nu_{3} \leq \nu_{2} \leq \nu_{4}$, 
        $(\Mse)^{\oplus \nu_4-\nu_3} \oplus (\Mei)^{\oplus \nu_4-\nu_2} \oplus (\Mte)^{\oplus \nu_2+\nu_3-\nu_1-\nu_4} \oplus (\Mel)^{\oplus \nu_1}$

        \item If $\nu_{3} \leq \nu_{1} \leq \nu_{4} \leq \nu_{2}$, 
        $(\Mtw)^{\oplus \nu_2+\nu_3-\nu_1-\nu_4} \oplus (\Mfi)^{\oplus \nu_1-\nu_3} \oplus (\Mse)^{\oplus \nu_4-\nu_3} \oplus (\Mel)^{\oplus \nu_3}$
    \end{enumerate}
    Note that because of the assumption $\nu_1+\nu_4 < \nu_2+\nu_3$, the cases $\nu_{3} \leq \nu_{4} \leq \nu_{2} \leq \nu_{1}$, $\nu_{3} \leq \nu_{2} \leq \nu_{4} \leq \nu_{1}$, $\nu_{3} \leq \nu_{1} \leq \nu_{2} \leq \nu_{4}$ and $\nu_{3} \leq \nu_{2} \leq \nu_{1} \leq \nu_{4}$ do not occur.

\end{thm}

\begin{proof}
    The representations listed above give required dimension vectors $(\nu_1,\nu_2,\nu_3,\nu_4)$ and ranks $(r_1,r_2)$. By Theorem \ref{CBS1} and \ref{CBS2}, it is enough to show that for each summand $M_{i}$ and $M_{j}$, we have $\textrm{Ext}^1(M_{i}, M_{j}) = \textrm{Ext}^1(M_{j}, M_{i})=0$ for $i\neq j$. This is directly checked by using projective resolutions of $M_i$ listed below, where 
    $\Mfo, \Mse, \Mei, \Mel$  are projective.

    \begin{align*}    
    0 \to \Mfo \to \Mse \oplus \Mei \to \Mel \to \Mon \to 0,& \\
    0 \to \Mfo \to \Mse \to \Mtw \to 0,& \\
    0 \to \Mfo \to \Mei \to \Mth \to 0,& \\
    0 \to \Mei \to \Mel \to \Mfi \to 0,& \\
    0 \to \Mse \to \Mel \to \Msi \to 0,& \\
    0 \to \Mfo \to \Mel \to \Mni \to 0,& \\
    0 \to \Mfo \to \Mse \oplus \Mei \to \Mte \to 0.& 
    \end{align*}    
\end{proof}

\bmhead{Author Contributions}
All authors equally contributed to the manuscript.

\bmhead{Data Availability}
Not Applicable.

\bmhead{Funding}
The authors are supported by the Japan Society for the Promotion of Science, Grant-in-Aid for Transformative Research Areas (A) (22A201,22H05107), JST MIRAI Program (JPMJMI22G1),
and JST CREST (JPMJCR2334).

\bmhead{Acknowledgements}
The authors thank the referee for valuable suggestions that helped improve the article and for pointing out errors in the manuscript. 

\section*{Declarations}

\bmhead{Competing interests}
The authors have no conflict of interest to declare that are relevant to this article. 

\bmhead{Ethical Approval} Not applicable.

\printbibliography

\end{document}

%% file: Definitions.tex
\usepackage{graphicx}
\usepackage{tikz-cd}
\usepackage{amssymb}
\usepackage{amscd}
\usepackage{amsmath}
\usepackage{amsthm}
\usepackage[noend]{algpseudocode}
\usepackage{algorithm}
\usepackage{enumerate}
\usepackage[all]{xy}
\usepackage{color}
\usepackage{ulem}
\usepackage{comment}
\usepackage{bm}

\allowdisplaybreaks

\newtheorem{thm}{Theorem}[section]
\newtheorem*{mainthm}{Main Theorem}
\newtheorem{lem}[thm]{Lemma}
\newtheorem{cor}[thm]{Corollary}
\newtheorem{prop}[thm]{Proposition}

\theoremstyle{definition}

\newtheorem{ex}[thm]{Example}
\newtheorem{defn}[thm]{Definition}
\newtheorem{rem}[thm]{Remark}

\newcommand\erase{\bgroup\markoverwith{\textcolor{red}{\rule[.5ex]{2pt}{0.4pt}}}\ULon}

\newcommand{\N}{{\mathbb{N}}}

\newcommand{\Z}{{\mathbb{Z}}}

\newcommand{\Hom}{{\rm Hom}}

\newcommand{\rank}{\mathop{\rm rank}}
\newcommand{\image}{\mathop{\rm im}}
\newcommand{\kernel}{\mathop{\rm ker}}
\newcommand{\coker}{\mathop{\rm coker}}

\newcommand{\weight}{{\rm wt}}

\newcommand{\Proj}{{\rm Proj}}

\newcommand{\irr}{\mathop{\rm Irr}}

\DeclareMathOperator{\GL}{GL}

\DeclareMathOperator{\Exttilde}{\widetilde{\rm Ext}}
\DeclareMathOperator{\out}{Out}
\DeclareMathOperator{\inward}{In}

\newcommand{\Mon}{
\begin{tikzpicture}
\coordinate (1) at (0,0);
\coordinate (2) at (0.18,0);
\coordinate (3) at (0,0.18);
\coordinate (4) at (0.18,0.18);
\filldraw[fill = black, draw=black] (1) circle (1pt);
\filldraw[fill = white, draw=black] (2) circle (1pt);
\filldraw[fill = white, draw=black] (3) circle (1pt);
\filldraw[fill = white, draw=black] (4) circle (1pt);
\end{tikzpicture}
}

\newcommand{\Mtw}{
\begin{tikzpicture}
\coordinate (1) at (0,0);
\coordinate (2) at (0.18,0);
\coordinate (3) at (0,0.18);
\coordinate (4) at (0.18,0.18);
\filldraw[fill = white, draw=black] (1) circle (1pt);
\filldraw[fill = black, draw=black] (2) circle (1pt);
\filldraw[fill = white, draw=black] (3) circle (1pt);
\filldraw[fill = white, draw=black] (4) circle (1pt);
\end{tikzpicture}
}

\newcommand{\Mth}{
\begin{tikzpicture}
\coordinate (1) at (0,0);
\coordinate (2) at (0.18,0);
\coordinate (3) at (0,0.18);
\coordinate (4) at (0.18,0.18);
\filldraw[fill = white, draw=black] (1) circle (1pt);
\filldraw[fill = white, draw=black] (2) circle (1pt);
\filldraw[fill = black, draw=black] (3) circle (1pt);
\filldraw[fill = white, draw=black] (4) circle (1pt);
\end{tikzpicture}
}

\newcommand{\Mfo}{
\begin{tikzpicture}
\coordinate (1) at (0,0);
\coordinate (2) at (0.18,0);
\coordinate (3) at (0,0.18);
\coordinate (4) at (0.18,0.18);
\filldraw[fill = white, draw=black] (1) circle (1pt);
\filldraw[fill = white, draw=black] (2) circle (1pt);
\filldraw[fill = white, draw=black] (3) circle (1pt);
\filldraw[fill = black, draw=black] (4) circle (1pt);
\end{tikzpicture}
}

\newcommand{\Mfi}{
\begin{tikzpicture}
\coordinate (1) at (0,0);
\coordinate (2) at (0.18,0);
\coordinate (3) at (0,0.18);
\coordinate (4) at (0.18,0.18);
\filldraw[fill = black, draw=black] (1) circle (1pt);
\filldraw[fill = black, draw=black] (2) circle (1pt);
\filldraw[fill = white, draw=black] (3) circle (1pt);
\filldraw[fill = white, draw=black] (4) circle (1pt);
\draw (1) -- (2);
\end{tikzpicture}
}

\newcommand{\Msi}{
\begin{tikzpicture}
\coordinate (1) at (0,0);
\coordinate (2) at (0.18,0);
\coordinate (3) at (0,0.18);
\coordinate (4) at (0.18,0.18);
\filldraw[fill = black, draw=black] (1) circle (1pt);
\filldraw[fill = white, draw=black] (2) circle (1pt);
\filldraw[fill = black, draw=black] (3) circle (1pt);
\filldraw[fill = white, draw=black] (4) circle (1pt);
\draw (1) -- (3);
\end{tikzpicture}
}

\newcommand{\Mse}{
\begin{tikzpicture}
\coordinate (1) at (0,0);
\coordinate (2) at (0.18,0);
\coordinate (3) at (0,0.18);
\coordinate (4) at (0.18,0.18);
\filldraw[fill = white, draw=black] (1) circle (1pt);
\filldraw[fill = black, draw=black] (2) circle (1pt);
\filldraw[fill = white, draw=black] (3) circle (1pt);
\filldraw[fill = black, draw=black] (4) circle (1pt);
\draw (2) -- (4);
\end{tikzpicture}
}

\newcommand{\Mei}{
\begin{tikzpicture}
\coordinate (1) at (0,0);
\coordinate (2) at (0.18,0);
\coordinate (3) at (0,0.18);
\coordinate (4) at (0.18,0.18);
\filldraw[fill = white, draw=black] (1) circle (1pt);
\filldraw[fill = white, draw=black] (2) circle (1pt);
\filldraw[fill = black, draw=black] (3) circle (1pt);
\filldraw[fill = black, draw=black] (4) circle (1pt);
\draw (3) -- (4);
\end{tikzpicture}
}

\newcommand{\Mni}{
\begin{tikzpicture}
\coordinate (1) at (0,0);
\coordinate (2) at (0.18,0);
\coordinate (3) at (0,0.18);
\coordinate (4) at (0.18,0.18);
\filldraw[fill = black, draw=black] (1) circle (1pt);
\filldraw[fill = black, draw=black] (2) circle (1pt);
\filldraw[fill = black, draw=black] (3) circle (1pt);
\filldraw[fill = white, draw=black] (4) circle (1pt);
\draw (1) -- (2);
\draw (1) -- (3);
\end{tikzpicture}
}

\newcommand{\Mte}{
\begin{tikzpicture}
\coordinate (1) at (0,0);
\coordinate (2) at (0.18,0);
\coordinate (3) at (0,0.18);
\coordinate (4) at (0.18,0.18);
\filldraw[fill = white, draw=black] (1) circle (1pt);
\filldraw[fill = black, draw=black] (2) circle (1pt);
\filldraw[fill = black, draw=black] (3) circle (1pt);
\filldraw[fill = black, draw=black] (4) circle (1pt);
\draw (2) -- (4);
\draw (3) -- (4);
\end{tikzpicture}
}

\newcommand{\Mel}{
\begin{tikzpicture}
\coordinate (1) at (0,0);
\coordinate (2) at (0.18,0);
\coordinate (3) at (0,0.18);
\coordinate (4) at (0.18,0.18);
\filldraw[fill = black, draw=black] (1) circle (1pt);
\filldraw[fill = black, draw=black] (2) circle (1pt);
\filldraw[fill = black, draw=black] (3) circle (1pt);
\filldraw[fill = black, draw=black] (4) circle (1pt);
\draw (1) -- (2);
\draw (1) -- (3);
\draw (2) -- (4);
\draw (3) -- (4);
\end{tikzpicture}
}

